\documentclass[a4paper, reqno,10pt]{amsart}

\usepackage{latexsym, amsfonts, amsmath, amsthm, amssymb}

\usepackage{graphics, graphicx}
\usepackage{a4wide}
\usepackage{tikz}

\newtheorem{thm}{Theorem}
\newtheorem{lemma}{Lemma}
\newtheorem{cor}{Corollary}
\newtheorem{prop}{Proposition}

\newtheorem{defin}{Definition}
\newtheorem{rem}{Remark}

\newcommand{\N}{\mathbb{N}}
\newcommand{\Z}{\mathbb{Z}}
\newcommand{\R}{\mathbb{R}}

\newcommand{\hs}{\overline{H}}
\newcommand{\eps}{\varepsilon}
\newcommand{\dotcup}{\,\ensuremath{\mathaccent\cdot\cup}\,}

\makeatletter
\def\@tocline#1#2#3#4#5#6#7{\relax
  \ifnum #1>\c@tocdepth 
  \else
    \par \addpenalty\@secpenalty\addvspace{#2}%
    \begingroup \hyphenpenalty\@M
    \@ifempty{#4}{%
      \@tempdima\csname r@tocindent\number#1\endcsname\relax
    }{%
      \@tempdima#4\relax
    }%
    \parindent\z@ \leftskip#3\relax \advance\leftskip\@tempdima\relax
    \rightskip\@pnumwidth plus4em \parfillskip-\@pnumwidth
    #5\leavevmode\hskip-\@tempdima
      \ifcase #1
       \or\or \hskip 1em \or \hskip 2em \else \hskip 3em \fi%
      #6\nobreak\relax
    \dotfill\hbox to\@pnumwidth{\@tocpagenum{#7}}\par
    \nobreak
    \endgroup
  \fi}
\makeatother

\makeatletter
\ifcsname phantomsection\endcsname
    \newcommand*{\qrr@gobblenexttocentry}[5]{}
\else
    \newcommand*{\qrr@gobblenexttocentry}[4]{}
\fi
\newcommand*{\addsubsection}{%
    \addtocontents{toc}{\protect\qrr@gobblenexttocentry}%
    \subsection}
\makeatother

\title{On Uniqueness of Weak Solutions for the Thin-Film Equation}
\author{Dominik John}
\address{Mathematisches Institut, Universit\"at Bonn, Endenicher Allee 60, 53115 Bonn, Germany}

\begin{document}

\begin{abstract}
In any number of space variables, we study the Cauchy problem related to the fourth order degenerate diffusion equation $\partial_s h + \nabla \!\cdot\! (h^m \nabla\Delta h)=0$ in the simplest case of a linearly degenerate mobility $m=1$. This equation, derived from a lubrication approximation, also models the surface tension dominated flow of a thin viscous film in the Hele-Shaw cell. Our focus is on uniqueness of weak solutions in the ``complete wetting'' regime, when a zero contact angle between liquid and solid is imposed. In this case, we transform the problem by zooming into the free boundary and look at small Lipschitz perturbations of a quadratic stationary solution. In the perturbational setting, the main difficulty is to construct scale invariant function spaces in such a way that they are compatible with the structure of the nonlinear equation. Here we rely on the theory of singular integrals in spaces of homogeneous type to obtain linear estimates in these functions spaces which provide ``optimal'' conditions on the initial data under which a unique solution exists. In fact, this solution can be used to define a class of functions in which the original initial value problem has a unique (weak) solution. Moreover, we show that the interface between empty and wetted regions is an analytic hypersurface in time and space.
\end{abstract}

\maketitle
\tableofcontents
\thispagestyle{empty}


\section{Introduction} 
\label{introduction}

\subsection{The model} 
\label{model}
In this paper, we will be discussing the Cauchy problem related to the fourth order degenerate equation
\begin{equation}
\label{tfe}
\tag{tfe}
\partial_s h + \nabla \!\cdot\! \bigl( h^m \, \nabla\Delta \, h \bigr) \, = \, 0 \, .
\end{equation}
Here $h: [0,\infty) \times \R^n \to \R$ is a real-valued function of time and space, the gradient and the Laplacian are in the space variables only, and $m>0$.

Possibly the simplest context in which equation (\ref{tfe}) applies is that of a thin liquid film spreading along a solid surface. For such problems, a lubrication approximation\footnote{The reference \cite{ODB97} and the references therein provide a comprehensive discussion on this.} simplifies the Navier-Stokes equations to
\[
\mu \, \partial_s h \, + \, \gamma \; \nabla \!\cdot\! \Bigl( \bigl( \frac{1}{3} \, h^3 + h^2 \, k(h) \bigr) \, \nabla\Delta \, h \Bigr) \, = \, 0 \, ,
\]
relating the fluid's velocity in the horizontal direction to its thickness and shape at the liquid-vapor interface. Here the constants $\mu$ and $\gamma$ denote the viscosity and the surface tension, respectively. Moreover, we assume that the mobility coefficient $k$ is a nonnegative function of $h$, vanishes at zero and has at most polynomial growth (say of order $m-2$). A weighted slip condition of the form
\[
v \, = \, k(h) \, \partial_z v \qquad \text{if} \; z=0 \, ,
\]
demanding that the horizontal fluid velocity $v$ is proportional to its normal derivative ($z$ is the coordinate perpendicular to the surface), entails $k(h) = b^{3-m} \, h^{m-2}$. For films sufficiently thinner than the slip length, i.e.\ $h \ll b$, the term $h^3$ can be neglected and equation (\ref{tfe}) holds for the dimensionless variables
\[
\hat{h} \, = \, \frac{h}{Z} \, , \; \hat{y} \, = \, \frac{y}{Y} \quad \text{and} \quad \hat{s} \, = \, \frac{\mu \, Z^m}{\gamma \, Y^4} \, b^{m-3} \, s \, ,
\]
with $Z>0$ being a characteristic film thickness and $Y>0$ a typical horizontal length scale. The underlying specific model determines the slip length $b \geq 0$ and the growth exponent $m$. Basically, one distinguishes two qualitatively different cases, weaker slippage ($m \in (2,3)$) and stronger slippage ($m \in (0,2)$). For the borderline case $m=2$, we recover the classical Navier slip condition. If $m=0$, then (\ref{tfe}) becomes non-degenerate at $\{ h=0 \}$, leading to the phenomenon of infinite speed of propagation. On the other hand, for $m \geq 3$, the support is expected to be constant in time (no slip). Indeed, the behavior of solutions of (\ref{tfe}) is subject to changes in the parameter regime $m \in (0,3)$.

We briefly compare the fourth-order degenerate diffusion equation (\ref{tfe}) to its second-order analogue
\begin{equation}
\label{pme}
\tag{pme}
\partial_s h \, = \, \Delta (h^m) \, , \quad m>1 \, ,
\end{equation}
the well-known porous medium equation. Both equations (\ref{tfe}) and (\ref{pme}) can be expressed by the conservation law
\[
\partial_s h + \nabla \!\cdot q(h) \, = \, 0 \, ,
\]
where $q(h)$ is the vector-valued flux of either an ideal gas in a porous medium or a thin layer of liquid deposited on a flat surface. In the absence of flows across the boundary of the support of $h$, that is
\[
q(h) \big\vert_{\partial \{ h>0 \}} \cdot \vec{\nu} \, = \, 0
\]
with $\vec{\nu}$ being the outer normal to $\partial \{ h>0 \}$, we expect conservation of mass: $\int_{\R^n} h(s,y) \, dy \equiv M$. For the model (\ref{pme}), a number of striking features are known by now including the following:
\begin{enumerate}
\item In regions of strictly positive $h$ solutions become instantaneously smooth.
\item There exists a maximum principle.
\item Compactly supported initial data generate compactly supported solutions at fixed times.
\item The Cauchy problem has a unique solution for a wide range of data including $L^1(\R^n)$.
\end{enumerate}
While the fact that both equations are diffusive guarantees that property (1) is preserved, a remarkable difference between (\ref{pme}) and (\ref{tfe}) is the lack of a maximum principle for the fourth-order equation. This can already be seen in the non-degenerate case $\partial_s h + \Delta^2 h = 0$, where solutions can change their sign. The remaining properties (3)--(4) have been subject of recent and ongoing study. We will discuss some examples in the following section.

A peculiarity of (\ref{tfe}) with linearly degenerate mobility $h$ is that its evolution can be understood as the gradient flow of the energy with respect to the $L^2$-Wasserstein metric. The corresponding energy is given by
\[
E(h) \, = \, \frac{1}{2} \int_{\{ h > 0 \}} ( |\nabla h|^2 + \theta^2) \, dy \, ,
\]
where the contact angle $\theta$ at the liquid-solid interface is determined by an equilibrium of surface energies. This is another similarity between the second-order and fourth-order equation since the porous medium equation also has gradient flow structure \cite{O01}.

\subsection{The background} 
\label{background}
The initial value problem related to (\ref{tfe}) admits the construction of nonnegative weak solutions - see Bernis and Friedman \cite{BF90}, and Bertozzi and Pugh \cite{BP96} for the case of one space dimension and Dal Passo et al.\ \cite{BDGG98,DGG98} for that one of multiple space dimensions in the parameter regime $m \in (\frac{1}{8}, 3)$.

Another remarkable observation is that property (3) holds for $m \in (0,3)$, that is, the support of solutions propagates with finite speed. In space dimension $n=1$, this has been established by Bernis \cite{B96a, B96b}, in the higher dimensional case it is derived in \cite{BDGG98, DGG98} for $m \in (0,2)$ and in \cite{G05} for $m \in [2,3)$.

More recent results on weak (or classical) solutions use the model to describe the dynamics of a moving interface in the context of either a complete wetting ($\theta=0$) or a partial wetting (fixed positive contact angle $\theta$). 

Giacomelli et al.\ \cite{GKO08} address the regularity theory for solutions in one space dimension with mobility $h$, when a zero contact angle is imposed. In particular, they study perturbations of a stationary solution near the free boundary and obtain maximal regularity for a linearized operator, and existence and uniqueness for the nonlinear problem. In fact, this solution is smooth up to the boundary of its support. For the Navier-slip model, a solution - a perturbation of a traveling wave - is actually non-smooth at the contact line, as can already be seen in the case of self-similar solutions (see \cite{GGO12}).

A treatment of the problem in the partial wetting regime may be found in \cite{O98, KM10}. Well-posedness under the classical Navier slip condition $m=2$ is discussed in \cite{K07}.

The methods used to prove results about the model proposed in (\ref{tfe}) rely on sophisticated energy arguments combined with weighted Sobolev estimates. In particular, entropy estimates play an important part in the development of an existence theory for weak solutions from nonnegative data. However, many questions are still unanswered. Probably the most important one lies in the classification of function classes in which (weak) solutions of (\ref{tfe}) are unique. On that subject, we refer to \cite{DGG98} for a counterexample.

\subsection{Outline} 
\label{outline}
In the present paper, we extend the results in \cite{GKO08} to the case of arbitrary space dimensions $n \in \N$ under conditions on the initial data that are optimal for the methods we use. The strategy is as follows: It is convenient to transform the evolution free boundary problem formulated in (\ref{tfe}) to an equivalent problem with fixed domain by formally interchanging the roles of the independent variable $y_n$ and the dependent variable $h$, a technique commonly known as von Mises transformation. In fact, this is the first time one applies this kind of transformation to a fourth-order parabolic equation. Using $(t, x)$ to denote the new independent variables and $u=u(t,x)$ to denote the dependent variable, we can write (\ref{tfe}) as the following initial value problem\footnote{All the nonlinear terms are collected in the inhomogeneity $f[u]$ - its precise form will be discussed in Lemma \ref{nonlinearity}.}:
\begin{equation}
\label{perturbation equation}
\tag{pe}
\partial_t u \, + \, Lu \, = \, f[u] \qquad \text{on} \quad (0,T) \times H \, ,
\end{equation}
with linear spatial part
\[
Lu \, = \, x_n^{-1} \Delta(x_n^{\,3} \, \Delta u) \, - \, 4 \, \Delta_{\R^{n-1}}u \, .
\]

In section \ref{transformation}, we will be concerned with the transformation of the equation onto the upper half space $\{ x_n > 0 \}$, here and in the following denoted by $H$.

Section \ref{linear model case} is dedicated to the analysis of the linear equation, where for the moment we ignore the dependency of $f$ on $u$. In particular, we prove a series of linear estimates for solutions using a rather weak solution concept on relatively open subsets of $\hs$. In the further course of section 3, we will be discussing a Gaussian estimate for the Green kernel and some of its useful consequences, allowing us to construct a solution of (\ref{perturbation equation}) via a fixed point argument, as discussed in section 4. Indeed, under very weak regularity assumptions on the initial data $g$, this solution is unique in an appropriate Banach space. Moreover, we use an argument introduced by Angenent \cite{A90}, and later improved by of Koch and Lamm \cite{KL12}, to obtain analyticity of the solution in time and all tangential directions up to the boundary of its support.

In section \ref{uniqueness, stability and regularity}, we prove the main result of this paper: The unique solution $u^*$ can be used to generate a solution $h$ of the thin-film equation on its positivity set $P(h) = \{ (s,y) \mid h(s,y) > 0\}$. It satisfies the identity
\begin{equation}
\label{weak solution of tfe}
\int_I \int_{\R^n} h \, \partial_s \varphi \, + \, h \, \nabla\Delta h \cdot\! \nabla\varphi \, dy ds \, = \, 0
\end{equation}
for all $\varphi \in C_c^{\infty}\bigl( (0,T) \times \R^n \bigr)$, that is, it is a weak solution of (\ref{tfe}). In fact, the solution obtained in this manner is unique and the expression
\[
\bigl[ \sqrt{h}\, \bigr]_{X_p^1} \, = \! \sup_{{s \in (0,T)}\atop{y\in P_s(h)}} \, |Q_{\!\sqrt[4]{s}\,}(y)|^{-\frac{1}{p}} \!\!\!\!\! \sum_{(j, l, \alpha) \in \mathcal{CZ}} \!\!\!\!\!\!\! \sqrt[4]{s}^{\,4l+|\alpha|-1} (\sqrt[4]{s}+\sqrt[4]{h(s,y)}\,)^{|\alpha|-2j-1} \| \sqrt{h}^{\,j} \partial_s^l \partial_y^\alpha \, \sqrt{h} \|_{L^p(Q_{\!\sqrt[4]{s}\,}(y))}
\]
is finite, where $(j, l, \alpha) \in \mathcal{CZ}$ means that the triple is admissible in a sense to be specified in (\ref{CZ exponents}). \\

\begin{thm}
\label{uniqueness}
Suppose $T>0$ and $\eps>0$ is sufficiently small. Given an initial datum $h(0)=h_0$ with
\[
|\nabla_{\!y}\sqrt{h_0(y)} - e_n| \, < \, \eps \, ,
\]
there exist a constant $c>0$ and a unique weak solution $h^* \in C\bigl( (0,T) \times \R^n \bigr)$ of (\ref{tfe}) with initial value $h_0$,
\[
\| \nabla_{\!y} \sqrt{h^*} - e_n \|_{L^{\infty}(P(h))} + \bigl[ \sqrt{h^*} \, \bigr]_{X_p^1} \, \leq \, c \, \eps \, ,
\]
and $h^*$ satisfies the equation (\ref{weak solution of tfe}). Moreover, the level sets of $h^*$ are analytic. In particular, this holds for the interface $\partial P(h^*)$ between empty and occupied regions.
\end{thm}
This implies large time stability of solutions that are initially close to the quadratic stationary solution, a possibly optimal result in terms of the regularity of the initial data. Moreover, these solutions are unique in the indicated class of functions and the (moving) interface is an analytic hypersurface in time and space.

To maintain a clear presentation, we defer most of the proofs and technical details to section \ref{proofs}.


\section{Transformation onto a fixed domain} 
\label{transformation}

To motivate the consideration of the equation (\ref{perturbation equation}), we will first compute the transformation of the thin-film equation (\ref{tfe}) when one interchanges independent and dependent variables near the boundary. As we suppose that any solution $h$ is nonnegative in its support this change of coordinates ensures that the original problem becomes a degenerate parabolic problem on a fixed domain, a method that has previously been used in the context of the porous medium equation (see \cite{DH98, K99}).

\subsection{Local coordinates} 
\label{local coords}
Assume that $h(s,y)$ is a solution of (\ref{tfe}) on $(0,T) \times \R^n$ with positivity set at  time $s$ given by
\[
P_s(h) \, = \, \{ y \mid h(s,y) > 0 \} \, \subsetneq \, \R^n \, .
\]
Pick a point $(s_0,y_0)$ on $\partial P(h)$ with $\partial_{y_n}h(s_0,y_0) = \eps_0 > 0$. Thus, by the implicit function theorem, we can solve the equation $z=h(s,y)$ locally near $(s_0,y_0)$ with respect to $y_n$ giving rise to a function
\[
y_n \, = \, v(s,y',z)
\]
defined for all $(s,y',z)$ in a small neighborhood of $(s_0,y'_0,0)$. We thus obtain $h(s,y',v(s,y',z))=z$ with
\[
\partial_s v \, = \, - (\partial_{y_n} h)^{-1} \partial_s h \, , \;  \nabla_{\!y'}v \, = \, - (\partial_{y_n} h)^{-1} \nabla_{\!y'}h \quad \text{and} \quad \partial_z v \, = \, (\partial_{y_n} h)^{-1} \, .
\]
This suggests a change of coordinates $(s,y) \mapsto (s',y',z) =: (t,x)$ which in turn implies the following relation:
\[
\partial_s h \, = \, - (\partial_{x_n} v)^{-1} \partial_t v \, , \;  \nabla_{\!y'}h \, = \, - (\partial_{x_n} v)^{-1} \nabla_{\!x'}v \quad \text{and} \quad \partial_{y_n} h \, = \, (\partial_{x_n} v)^{-1} \, .
\]
Also note that the free boundary $\{ h=0 \}$ has now been transformed into the fixed boundary $\partial H$. We can differentiate with respect to the $y$-variable to compute the second order derivatives. We get
\[
\Delta'_y h \, = \, - (\partial_{x_n}v)^{-1} ( \Delta'_x v - \partial_{x_n} \left( \frac{|\nabla'_{\!x}v|^2}{\partial_{x_n}v} \right) ) \quad \text{and} \quad \partial_{y_n}^2 h \, = \, - \, \frac{\partial_{x_n}^2 v}{(\partial_{x_n}v)^3} \, ,
\]
where $\Delta' = \Delta_{\R^{n-1}}$ is the $(n-1)$-dimensional Laplacian leaving aside the $x_n$-direction. In similar fashion, we obtain
\[
\nabla_{\!y} \, = \, 
	\begin{pmatrix} 
	\nabla'_{\!x} - (\partial_{x_n} v)^{-1} \nabla'_{\!x}v \, \partial_{x_n} \\
	(\partial_{x_n} v)^{-1} \, \partial_{x_n} 
	\end{pmatrix}
\]
by the chain rule. Here, $\nabla'_{\!x} = \nabla_{\!x'}$ denotes the gradient in the variable $x'=(x_1, \dots, x_{n-1})$ and therefore is an $(n-1)$-dimensional vector.

\subsection{Perturbations of stationary solutions} 
\label{linearization}
We assume that the profile of a solution is approximately quadratic, i.e.\ $h_0 \sim dist(y, \R^n \setminus P_0(h))^2$. In this case we set $\tilde{h} = \sqrt{h}$ and rewrite (\ref{tfe}) in the form
\[
\partial_s \tilde{h}^2 + \nabla\!\cdot\! \bigl(\tilde{h}^2 \; \nabla\Delta \tilde{h}^2\bigr) = 0 \, ,
\]
or equivalently
\begin{equation}
\label{tfe^2}
\begin{split}
\partial_s \tilde{h} \, 
&+ \, \nabla\!\cdot\! \bigl(\tilde{h}^2 \, \nabla\Delta \tilde{h}\bigr) \, + \, 4 \, \tilde{h} \nabla \tilde{h} \cdot\! \nabla\Delta \tilde{h} \, + \, \tilde{h} (\Delta \tilde{h})^2 \, + \, 2 \, \tilde{h} |D_{\!y}^2 \tilde{h}|^2 \, +  \\
&+ \, 2 \, |\nabla \tilde{h}|^2 \Delta \tilde{h} \, + \, 4 \sum_{i,j=1}^n (\partial_{y_i} \tilde{h}) (\partial_{y_j}\tilde{h}) \, \partial_{y_i y_j} \tilde{h} \, = \, 0 \, .
\end{split}
\end{equation}
\textbf{Remark:} We only consider the possibly simplest case of a linearly degenerate mobility $m=1$. \\

Now as a reference frame for (\ref{tfe}) we take the stationary solution $h_{st}(y) = (y_n)_+^{\,2}$, corresponding to $\tilde{h}_{st}(y)=(y_n)_+$ for (\ref{tfe^2}). In the new coordinates, this becomes $v_{st}(x) = x_n$ for $x \in H$. We are now interested in the perturbed steady state $v(t,x) = v_{st}(x) + u(t,x)$, where $u$ can be regarded as a small perturbation. Via the above change of coordinates (with $h$ replaced by $\tilde{h}$), equation (\ref{tfe}) - or rather (\ref{tfe^2}) - transforms into (\ref{perturbation equation}), where all the nonlinear terms are collected in $f[u]$. Its precise form will be discussed in Lemma \ref{nonlinearity}. Figure \ref{fig:perturbation} captures the situation $a)$ before and $b)$ after these transformations.

\begin{figure}[htb]
\centering
\begin{tikzpicture}[domain=0:2] 

\draw[->, thick] (-1,0) -- (5,0) node[below right] {$y_n$};
\draw[->, thick] (0,0) -- (0,2.5) node[left] {$h$};
\draw[thick] (-1,0) -- (0,0) parabola (4.5,2.25);
\node at (2.85,.5) {$h_{st}$};
\draw[thick] (-.5,0) parabola (1.5,1.44) to [out=40,in=180] (2.8,1.6) parabola (4.5,2.35);
\node at (.9,1.2) {$h_0$};
\node at (2,-.75) {\small{\it{a)} $h_0 \approx h_{st}$}};

\draw[->, thick] (6.5,0) -- (11.5,0) node[below right] {$x_n$};
\draw[->, thick] (6.5,-.5) -- (6.5,2.5) node[left] {$v$};
\draw[thick] (6.5,0) to (11,2.25);
\node at (8.25,1.2) {$v_{st}$};
\draw[thick] (6.5,-.25) to (10,.7) to [out=10,in=225] (11.1,2.25);
\draw[->](9.8,1.27) -- (9.8,1.6);
\node at (9.8,1.15) {$u$};
\draw[->](9.8,1.03) -- (9.8,.7);
\node at (9,-.75) {\small{\it{b)} $v=v_{st}+u$}};

\end{tikzpicture} 
\vspace{-1ex}
\caption{Changing independent and dependent variables.}
\label{fig:perturbation}

\end{figure}
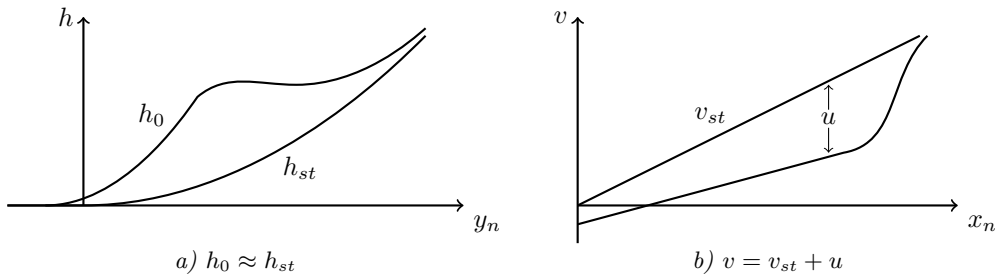

In order to give f[u] a more convenient form we first introduce the $\star$ notation to denote an arbitrary linear combination of products of indices for derivatives of $u$. For example, we write $\nabla u \star D_{\!x}^2 u$ for both $\partial_{x_n}u \Delta u$ and $\nabla'u \cdot \nabla'\partial_{x_n}u$. Moreover, we abbreviate the iterated application of $\star$ by
\[
P_j(D_{\!x}^k u) \, = \, \underbrace{D_{\!x}^k u \star \dots \star D_{\!x}^k u}_{j-\text{times}} \, , \qquad j,k \in \N_0 \, .
\]

\begin{lemma}
\label{nonlinearity}
Equation (\ref{tfe}) can be transformed into $\partial_tu + Lu = f_0[u] + x_n f_1[u] + x_n^{\,2} \, f_2[u]$ on $I \times H$ with
\[
\begin{split}
f_0[u] \, &= \, f_0^1(\nabla u) \star \nabla u \star D_{\!x}^2 u \, , \\
f_1[u] \, &= \, f_1^1(\nabla u) \star \nabla u \star D_{\!x}^3 u + f_1^2(\nabla u) \star P_2(D_{\!x}^2 u) \quad \text{and} \\
f_2[u] \, &= \, f_2^1(\nabla u) \star \nabla u \star D_{\!x}^4 u + f_2^2(\nabla u) \star D_{\!x}^2 u \star D_{\!x}^3 u + f_2^3(\nabla u) \star P_3(D_{\!x}^2 u) \, ,
\end{split}
\]
and $Lu = x_n^{-1} \Delta(x_n^{\,3} \, \Delta u) \, - \, 4 \, \Delta'u$.
\end{lemma}

\begin{proof}
Considering equation (\ref{tfe^2}) we first transform each summand separately and then linearize the transformed terms around the steady state $v_{st}(x)=x_n$. For the temporal part we get $\partial_s \tilde{h} = - \frac{\partial_tu}{v_n}$, where $v_n = \partial_{x_n} v = 1+\partial_{x_n} u$. Using the transformation formulas from section \ref{local coords}, we calculate term by term.
\[
- v_n \, \nabla \!\cdot\! ( \tilde{h}^2 \, \nabla \Delta \tilde{h}) \, = \, \nabla \!\cdot\! ( x_n^{\,2} \, \nabla \Delta u) - R_1(u) \, ,
\]
where
\[
R_1(u) \, = \, x_n f_1[u] + x_n^{\,2} f_{2}[u] \, .
\]
Moreover,
\[
- 4 \, v_n \, \tilde{h} \, \nabla \tilde{h} \cdot\! \nabla\Delta \tilde{h} \, = \, 4 \, x_n \, \partial_{x_n}\Delta u - R_{2}(u) \, ,
\]
where this time the remainder reduces to $R_2(u) = x_n f_1[u]$. For the next expression, we argue similarly to get
\[
- 2 \, v_n \, |\nabla \tilde{h}|^2 \, \Delta \tilde{h} \, = \, 2 \, \Delta u - f_0[u] \, .
\]
The last term of (\ref{tfe^2}) transforms to 
\[
- 4 \, v_n \sum_{i,j=1}^n (\partial_{y_i} \tilde{h}) (\partial_{y_{j}}\tilde{h}) \, \partial_{y_i y_j} \tilde{h} \, = \, 4 \, \partial_{x_n}^2 u - f_0[u] \, .
\] 
It remains to check the quadratic expressions in (\ref{tfe^2}). These terms, however, do not contain any linear parts such that both are completely absorbed by the inhomogeneity. Indeed, one can prove that
\[
- v_n \bigl( \tilde{h} \, (\Delta \tilde{h})^2 + 2 \, \tilde{h} \, |D_{\!y}^2 \tilde{h}|^2 \bigr) \, = \, - x_n f_1[u]
\]
in the same manner as above. Altogether this finishes the proof of the lemma.
\end{proof}

\begin{rem}
\label{decomposition of f_i^k}
A more detailed computation of the inhomogeneity $f[u]$ shows that the $f_i^k(\nabla u)$ in Lemma \ref{nonlinearity} decompose into factors of the form $P_j(\nabla u)$ and  $(1+\partial_{x_n} u)^{-m}$, for some combinations of the integers $0 \leq j \leq 4$ and $1 \leq m \leq 6$. We shall use this fact later on when it comes to proving the nonlinear estimates in section \ref{nonlinear estimates}.
\end{rem}

For $0<T\leq\infty$, let $I=(0,T)$ and $\Omega \subseteq H$ be open. We begin by discussing the linear equation $\partial_t u + Lu=f$ on $I \times \Omega$, that is, we ignore the dependency of $f$ on $u$. A distributional solution is characterized by
\[
- u(x_n \, \partial_t \varphi) + \Delta u(x_n^{\,3} \, \Delta \varphi) + 4 \, \nabla' u(x_n \nabla' \varphi) \, = \, f(x_n \, \varphi)
\]
for all test functions $\varphi \in C_c^{\infty}(I\times\Omega)$. If $u, \nabla'u, \Delta u, f \in L_{loc}^1(I\times\Omega)$, then this translates into
\begin{equation}
\label{distributional solution}
- \int_{I\times\Omega} x_n \, u \, \partial_t \varphi \, dx dt + \int_{I\times\Omega} x_n^{\,3} \, \Delta u \, \Delta \varphi \, dx dt + 4 \int_{I\times\Omega} x_n \nabla' u \cdot \nabla' \varphi \, dx dt \, = \, \int_{I\times\Omega} x_n \, f \, \varphi \, dx dt \, .
\end{equation}
Formally, we obtain this weak formulation by integrating the equation against a suitable test function $\varphi$ and then applying integration by parts with respect to time and space, respectively. In such a way, we can relax the requirements on the differentiability of solutions considerably, and thus we will use the integral representation (\ref{distributional solution}) as a model for our definition of solution. \\

The following scaling behavior turns out to be crucial: If $u$ is a solution of the linear equation on $I\times\Omega$, then
\begin{equation}
\label{linear scaling}
u_\lambda(t,x) \, = \, \lambda^{-2} (u \circ T_\lambda)(t,x) \, = \, u(\lambda^2 t, \lambda x)
\end{equation}
is a solution on $T_\lambda^{-1}(I\times\Omega)$ to the inhomogeneity $f \circ T_\lambda$. Moreover, given an initial value $u(0)=g$, we have $u_\lambda(0, \cdot) = \lambda^{-2} g(\lambda \, \cdot)$. The nonlinear equation, on the other hand, is invariant under $T_\lambda$ in the sense that
\[
(\partial_t + L) u_\lambda \, = \, \lambda (\partial_{\hat{t}} + \widehat{L})u \, = \, f[u] \, = \, \lambda^{-1} f[u_\lambda] \, ,
\]
where $(\hat{t}, \hat{x}) = (\lambda^2 t, \lambda x)$ and $\widehat{L}$ is the spatial linear operator with respect to the variable $\hat{x}$. The last equality can be verified by means of the formula $\partial_x^\alpha u_\lambda = \lambda^{|\alpha|-2} \partial_{\hat{x}}^\alpha u$. 


\section{The model linear degenerate equation} 
\label{linear model case}

\subsection{Preliminaries} 
\label{preliminaries}
Our aim in this section is to prove a Gaussian estimate for the linear degenerate evolution equation
\[
\partial_t u + Lu \, = \, f \, , \quad u(0,\cdot) \, = \, g \, ,
\]
on the half space $H$. The diffusion $L$ is governed by a Riemannian metric $g$ given by the family of inner products
\[
g_x(v,w) \, = \, x_n^{-1} \, v \cdot w
\]
on the tangent space $T_x H \cong \R^n$ that is attached to $x \in H$, where by $v \cdot w$ we denote the standard scalar product on $\R^n$. The Riemannian structure allows to measure the length of parametrized curves $\gamma \subset H$ by 
\[
\ell_g(\gamma) \, = \, \int_a^b \sqrt{g_{\gamma(s)}\bigl( \gamma'(s), \gamma'(s) \bigr)} \, ds \, ,
\]
and hence induces an intrinsic metric $d$ called the metric of Carnot-Caratheodory. The distance between two points $x$ and $y$ in this metric can be extended to the boundary $\partial H$ and is equivalent to the expression
\[
\rho(x,y) \, = \, \frac{|x-y|}{x_n^{\,\frac{1}{2}} + y_n^{\,\frac{1}{2}} + |x-y|^\frac{1}{2}}
\]
in the sense that 
\begin{equation}
\label{equivalence}
c_d^{-1} d(x,y) \, \leq \, \rho(x,y) \, \leq \, d(x,y)
\end{equation}
for a constant $c_d > 1$. Let us define the intrinsic ball, or $d$-ball, of radius $R>0$ centered at $x \in \hs$ by the set
\[
B_R(x) \, = \, \{ y \in \hs \mid d(x,y) < R \} \, ,
\]
and the Euclidean ball (intersected with $\hs$) by $B_R^{eu}(x) = \{ y \in \hs \mid |x-y| < R \}$. A natural choice of measure in this context is $d\mu_\sigma = x_n^{\,\sigma} dx$ for $\sigma > -1$. We have $x_n^{\,\sigma} \in L_{loc}^1(\hs)$ if and only if $\sigma > -1$, and hence by
\[
|\Omega|_\sigma \, = \, \mu_\sigma(\Omega) \, = \, \int_\Omega x_n^{\,\sigma} \, dx \, , \quad \Omega \subseteq \hs\, ,
\]
a Radon measure is canonically associated with the weight $x_n^{\,\sigma}$. For the Lebesgue measure (the case $\sigma=0$), we write $\mathcal{L}^n$ instead of $\mu_0$ and $|\Omega|$ instead of $|\Omega|_0$. \\
For every $\sigma > -1$, the measure $\mu_\sigma$ satisfies a doubling condition with respect to the metric $d$, that is, we have
\begin{equation}
\label{doubling condition}
0 \, < \, |B_{cR}(x)|_\sigma \, \leq \, b |B_R(x)|_\sigma \, < \, \infty
\end{equation}
for some constants $c,b\geq 1$. Hence the metric measure space $(\hs, d, \mu_\sigma)$ defines a space of homogeneous type, a setting in which the Calder\'on-Zygmund theory can be established. Also note that $\mu_\sigma$ has the same collection of null sets as the Lebesgue measure (in symbols $\mu_\sigma \sim \mathcal{L}^n$).

For local considerations, we will also need the following properties: Using the above notations we have
\begin{equation}
\label{ball topology}
B_{c_d^{-2} R(R+\sqrt{x_n})}^{eu}(x) \, \subset \, B_R(x) \, \subset \, B_{2 R(R+2\sqrt{x_n})}^{eu}(x) \, ,
\end{equation}
and as a consequence
\begin{equation}
\label{ball measure}
|B_{R}(x)|_\sigma \, \sim \, R^n (R + \sqrt{x_n})^{n+2\sigma} \, .
\end{equation}
For the parabolic equation we use the parabolic distance function $d^{(t)}: \R \times \hs \times \R \times \hs \to [0,\infty)$ with
\[
d^{(t)}\bigl( (t,x), (s,y) \bigr) \, = \, \sqrt[4]{|t-s| + d(x,y)^4} \, ,
\]
and accordingly, $(\R\times \hs, d^{(t)}, \mathcal{L} \times \mu_\sigma)$ is a homogeneous type space corresponding to the linear parabolic equation $\partial_t u + Lu = f$. For more details and the proofs of (\ref{equivalence})--(\ref{ball measure}), we refer the reader to \cite{DH98, K99}.\\

The Lebesgue space with respect to $\mu_\sigma$ on $\Omega \subseteq \hs$ is denoted by $L^p(\Omega, \mu_\sigma)$. Since $\mu_\sigma \sim \mathcal{L}^n$, we have
\[
L^{\infty}(\Omega, \mu_\sigma) \, = \, L^{\infty}(\Omega, \mathcal{L}^n) \, .
\]
We sometimes abbreviate the notation to $L^p(\mu_\sigma)$ if $\Omega=H$ and to $L^p(\Omega)$ if $\sigma=0$. When $1 \leq p < \infty$, the dual space of $L^p(\Omega, \mu_\sigma)$ has a natural isomorphism with $L^{p'}(\Omega, \mu_\sigma)$, where $p'$ is the conjugate exponent satisfying $\frac{1}{p} + \frac{1}{p'} =1$, which associates $v \in L^{p'}(\Omega, \mu_\sigma)$ with the functional $J_v$ given by
\[
J_v: L^p(\Omega, \mu_\sigma) \ni u \mapsto \int_\Omega u v \, d\mu_\sigma \, .
\]
Another important property of (weighted) Lebesgue spaces is that $C_c^{\infty}(\Omega)$ is a dense subspace for $p \in [1,\infty)$ (for $p=\infty$ this is false). \\
Now suppose $\Omega = \hs$. Then the same density result holds true for weighted Sobolev spaces defined by
\[
W^{k, p}(\Omega, \mu_{\sigma_0}, \dots, \mu_{\sigma_k}) \, = \, \{ u \in L^p(\Omega, \mu_{\sigma_0}) \mid \partial_{x}^\alpha \, u \in L^p(\Omega, \mu_{\sigma_{|\alpha|}}) \; \forall \, |\alpha| \leq k \} \, .
\]
More precisely,
\begin{equation}
\label{density}
\overline{C_c^{\infty}(\hs)} \, = \, W^{2,2}(H, \mu_1, \mu_1, \mu_3)
\end{equation}
provided that $\sigma_k \geq \dots \geq \sigma_0 > -1$ with $\sigma_{j+1}-\sigma_j \leq p$ and $\sigma_k \geq 0$, see \cite{K85}. Here the closure refers to the norm
\[
\| u \|_{W^{k, p}(H, \,\mu_{\sigma_0}, \dots, \mu_{\sigma_k})} \, = \, \Bigl( \sum_{|\alpha| \leq k} \| \partial_x^\alpha u \|_{L^p(H,\mu_{\sigma_{|\alpha|}})}^p \Bigr)^\frac{1}{p} \, .
\]
This density result will be very useful for proving certain inequalities which only involve Sobolev norms. In such situations it suffices to prove the inequalities only for functions in $C_c^{\infty}(\hs)$ (with classical derivatives instead of distributional ones). For example, applying integration by parts repeatedly, we get
\begin{equation}
\label{auxiliary formula}
\| D_{\!x}^2 \varphi \|_{L^2(\mu_3)}^2 \, = \, \| \Delta \varphi \|_{L^2(\mu_3)}^2 + 6 \, \| \nabla' \varphi \|_{L^2(\mu_1)}^2
\end{equation}
for any $\varphi \in C_c^{\infty}(\hs)$. By approximation, the same identity also holds for all $u \in W^{2,2}(H, \mu_1, \mu_1, \mu_3)$.

\subsection{Energy estimates} 
\label{energy estimates}
Considering the linear equation one can in general not expect to have classical solutions. Indeed, not even the solution concept of a distributional solution that we have introduced in section \ref{linearization} satisfies our requirements. However, we take the integral representation (\ref{distributional solution}) as a starting point (or motivation) and search for admissible extensions to relatively open subsets of $\hs$. We will call a solution to be legitimate if it retains a certain behavior towards the boundary in the sense that it admits values at $\partial H$. For this it is important to choose the test function space properly. More precisely, a test function is sometimes supposed to vanish at initial time depending on whether we consider the initial value problem or not.

\begin{defin}
\label{energy solution}
Let $I = (0,T)$ be open, $\Omega \subseteq \hs$ relatively open and $f\in L_{loc}^1\bigl( [0, T); L^2(\Omega,\mu_1) \bigr)$.
\begin{enumerate}
\item[i)]
We call $u$ an energy solution of $\partial_{t}u+Lu=f$ on $I\times\Omega$, if and only if $u\in L_{loc}^2\bigl( I; L^2(\Omega,\mu_1)\bigr)$, $\nabla u\in L^2\bigl( I; L^2(\Omega,\mu_1) \bigr)$, $D_{\!x}^2 u \in L^2\bigl( I; L^2(\Omega,\mu_3)\bigr)$, and the identity (\ref{distributional solution}) holds for all test functions $\varphi \in C_c^{\infty}(I \times \Omega)$.
\item[ii)]
We call $u$ an energy solution of the initial value problem on $[0,T)\times\Omega$ with $u(0)=g \in L^2(\Omega,\mu_1)$, if and only if $u\in L_{loc}^2\bigl( [0, T); L^2(\Omega,\mu_1)\bigr)$, $\nabla u$ and $D_{\!x}^2 u$ satisfy the assumptions of (i), and (\ref{distributional solution}) holds for all $\varphi \in C_c^{\infty}\bigl( [0, T) \times\Omega \bigr)$ with the additional term $\int_\Omega g \,\varphi(0) \, d\mu_1$ added to the right hand side of identity (\ref{distributional solution}).
\end{enumerate}
\end{defin}

A solution is constructed by the Galerkin method: We choose a suitable increasing sequence of finite dimensional subspaces $V_N \subset W^{2,2}(H, \mu_1, \mu_1, \mu_3)$ and look for functions $u_N: I \to V_N$ for which the weak formulation of the equation projected onto $V_N$ holds. The problem reduces to an ODE which can be solved by standard ODE theory. Each of the $u_N$ satisfies an a~priori estimate, the so-called energy estimate, which allows us to pass to the limit $N\to\infty$. This yields a weak solution in the sense of Definition \ref{energy solution} (ii).

\begin{prop}
\label{energy identity}
Let $I=(0,T)$ be open and $f \in L_{loc}^1\bigl( [0,T); L^2(\mu_1) \bigr)$. If $u$ is an energy solution of $\partial_t u + Lu = f$ on $I \times \hs$, then $u \in C\bigl( [s,T), L^2(\mu_1) \bigr)$ for any $s \in I$ and we have the energy identity
\[
\frac{1}{2} \, \| u(t) \|_{L^2(\mu_1)}^2 + \int_s^t \| \Delta u \|_{L^2(\mu_3)}^2 + 4 \, \| \nabla' u \|_{L^2(\mu_1)}^2 \, d\mathcal{L} \, = \, \frac{1}{2} \, \| u(s) \|_{L^2(\mu_1)}^2 + \int_s^t (f \mid u)_{L^2(\mu_1)} \, d\mathcal{L}
\]
for any $s < t \leq T$. If additionally $g \in L^2(\mu_1)$ is given, then there exists an energy solution $u \in C\bigl( \bar{I}, L^2(\mu_1) \bigr)$ with $u(0)=g$ and the energy identity also holds for $s = 0$.
\end{prop}

\begin{proof}
First note that the density result of section \ref{preliminaries} allows us to use modifications of the solution itself as test functions. In particular, the energy identity follows directly from a formal calculation using $\varphi = \chi_{(s,t)} u$ as test function. In order to make this a legitimate test function we need to approximate the characteristic function and to regularize $u$ with respect to time. More precisely, we use
\[
\varphi(\tau) \, = \, \chi_\delta(\tau) \, \eps^{-1} \int_\tau^{\tau+\eps} u \, d\mathcal{L}
\]
in the weak formulation of the linear equation (regularized in the same manner) and send $\delta, \eps \to 0$. This gives the energy identity for every energy solution. Here, the approximate characteristic function is given by
\[
\chi_\delta(\tau) \, = \, \frac{1}{\delta} \;
\begin{cases}
0 & \text{if}\quad 0 < \tau \leq s \\
\tau-s & \text{if} \quad s < \tau \leq s + \delta \\
\delta & \text{if} \quad s + \delta < \tau \leq t \\
t-\tau + \delta& \text{if} \quad t < \tau \leq t + \delta \\
0 & \text{if} \quad t + \delta < \tau < T
\end{cases}
\]
for a sufficiently small $\delta > 0$. As a consequence, we see that $t \mapsto \| u(t) \|_{L^2(\mu_1)}$ is continuous. By construction, $t \mapsto u(t) \in L^2(\mu_1)$ is weakly continuous. This implies the continuity of $t \mapsto u(t)$ in the Hilbert space $L^2(\mu_1)$. For the initial value problem, we proceed similarly.
\end{proof}

\begin{rem}
It is essential in the proof of Proposition \ref{energy identity} to consider a solution on the whole half space $\hs$, because only here we have the density result stated in (\ref{density}). For general local solutions, the energy identity is false.
\end{rem}

\begin{cor}
An energy solution of the Cauchy problem on $I \times \hs$ is uniquely determined by $f$ and $g$: If $u_1$ and $u_2$ are two solutions with $u_1(0)=u_2(0)=g$, then $u=u_1-u_2$ satisfies the homogeneous equation with zero Cauchy data. Proposition \ref{energy identity} now implies $\| u(t) \|_{L^2(\mu_1)} = 0$, and hence $u \equiv 0$.
\end{cor}

We wish to show that the operator $f \mapsto u$ is a singular integral operator defined on certain weighted Sobolev spaces. This requires two types of estimates: kernel estimates and a pointwise estimates. In preparation for the latter, we prove an energy-type estimate.

\begin{prop}
\label{energy estimate}
For $T>0$, let $I = (0,T) \subset \R$ be an open interval. If $f \in L^2\bigl( I; L^2(\mu_1) \bigr)$, then there exists a positive constant $c$ only depending on space dimension $n$ such that the energy estimate
\[
\int_I \| \partial_t u \|_{L^2(\mu_1)} + \| D_{\!x}^2 u \|_{W^{2,2}(\mu_1, \mu_3, \mu_5)} \, d\mathcal{L} \, \leq \, c \int_I \| f \|_{L^2(\mu_1)} \, d\mathcal{L}
\]
holds for the unique energy solution $u$ of $\partial_t u + Lu=f$ on $I \times \hs$ with $u(0)=0$.
\end{prop}

We next state a sort of analogue to the global energy estimate for local solutions, i.e.\ energy solutions defined in
\[
Q_R(t,x) \, = \, (t-R^4, t] \times B_R(x) \, ,
\]
the half-closed circular cylinder of radius $R>0$, top center $(t,x)$ and height $R^4$. In the following we refer to $Q_R(t,x)$ simply as (fourth-order) parabolic cylinder.

\begin{rem}
\label{local energy estimate}
If $u$ solves the homogeneous equation on $Q_R(t_0,x_0)$, then there exists an $\eps > 0$ such that
\[
\| \partial_t^l \partial_x^\alpha u \|_{L^2(Q_{\eps R}(t_0, x_0)), \mu_1} \, \leq \, c(n,l,\alpha) \, R^{-4l-|\alpha|} (R + \sqrt{x_{0,n}})^{-|\alpha|} \, \| u \|_{L^2(Q_R(t_0,x_0), \mu_1)}
\]
for all $l \in \N_0$ and all multi-indices $\alpha$. The proof essentially relies on the scaling behavior of the linear equation. Indeed, scaling reduces\footnote{For more details see the proof of Proposition \ref{pointwise estimate} in the appendix, section \ref{proofs}.} the estimate to $Q_R(0, e_n)$ with $R \ll 1$ or $Q_1(0,0)$.
\end{rem}

The ``derivative factor'' involving $R$ and $x_{0,n}$ is an important quantity. In the following we will abbreviate it by
\begin{equation}
\label{derivative factor}
\delta_{l, \alpha}(R, x_0) \, = \, R^{-4l-|\alpha|} (R + \sqrt{x_{0,n}})^{-|\alpha|} \, .
\end{equation}

\subsection{Pointwise estimates} 
\label{pointwise estimates}
Our goal here is to prove a pointwise estimate for solutions of the homogeneous equation on the cylinder $Q_R(t_0, x_0)$. At the same time, this shows that any local solution is indeed smooth, at least on a smaller cylinder.

\begin{prop}
\label{pointwise estimate}
Suppose $u$ satisfies the equation $\partial_t u + Lu = 0$ on $Q_R(t_0, x_0)$ in the energy sense for some $(t_0, x_0) \in \R \times \hs$. Then for any $l\in\N_{0}$ and any multi-index $\alpha$ there exists $c=c(n,l,\alpha)$ such that
\[
|\partial_t^l \partial_x^\alpha u(t_0, x_0)| \, \leq \, c \; \frac{\delta_{l, \alpha}(R, x_0)}{R^2 \, |B_R(x_0)|_1^{\,\frac{1}{2}}} \; \| u \|_{L^2(Q_R(t_0, x_0), \, \mu_1)} \, .
\]
\end{prop}

We now consider the initial value problem on the whole half space again. Clearly, any such solution on $I \times \hs$ is also a solution on $Q_R(t_0, x_0)$ for any $x_0 \in \hs$ provided we have $(t_0-R^4, t_0] \subset I$.

\begin{cor}
\label{pointwise estimate by initial data}
For $T>0$ let $I = (0, T)$ be an open interval, $l\in\N_0$ and $\alpha$ be a multi-index. We further suppose that $u$ is an energy solution of $\partial_t u + Lu=0$ on $[0,T) \times \hs$ with $u(0) = g \in L^2(H, \mu_1)$. Then we have
\[
|\partial_t^l \partial_x^\alpha u(t,x)| \, \leq \, c(n,l,\alpha) \, \delta_{l,\alpha}(\sqrt[4]{t},x) \, |B_{\!\sqrt[4]{t}}(x)|_1^{-\frac{1}{2}} \, \| g \|_{L^2(\mu_1)}
\]
for all $(t,x) \in I \times \hs$.
\end{cor}

\begin{proof}
The estimate follows from Proposition \ref{pointwise estimate} with $R = \sqrt[4]{t}$ and Propostion \ref{energy identity}.
\end{proof}

The crucial ingredient in this proof has been the property that $\| u(t) \|_{L^2(\mu_1)}$ decreases in $t$. In the next step we want to extend this norm decrease of solutions by an exponential function depending on the intrinsic metric $d$ that has been introduced in section \ref{preliminaries}.

\begin{lemma}
\label{exponential norm decrease}
Suppose $\Psi \in C^2(H)$ satisfies the conditions $\sqrt{x_n} \, |\nabla \Psi(x)| \leq c_L$ and $x_n |D_{\!x}^2 \Psi(x)| \leq c_L$ for all $x \in H$. If $u$ is an energy solution of $\partial_t u + Lu=0$ on $[0,T) \times \hs$ with $u(0) = g \in L^2(\mu_1)$, then we have
\[
\| e^\Psi u(t) \|_{L^2(\mu_1)} \, \leq \, c \, e^{c_n (c_L^{\,2} + c_L^{\,4}) t} \, \|e^\Psi g \|_{L^2(\mu_1)}
\]
for all $t \in [0,T]$.
\end{lemma}

To obtain a function $\Psi$ with the required properties, one has to construct a suitable approximation of the metric $d$ in terms of the equivalent quasimetric $\rho$ and use the relation (\ref{equivalence}).

\begin{rem}
\label{Lipschitz condition}
Suppose the indicated condition $x_n |\nabla \Psi(x)|^2 \leq c_L^{\,2}$ holds for all $x \in H$. Further let $\gamma: [a,b] \to H$ be the geodesic between the two points $x$ and $y$ which is parameterized by arc length, i.e.\ we may assume that $b=a+d(x,y)$ and $\gamma_{n}^{-1} |\gamma'|^2 \equiv 1$. By the fundamental theorem of calculus we then get
\[
|\Psi(x) - \Psi(y)| \, = \, \bigl| \int_a^b \nabla \Psi\bigl(\gamma(s)\bigr) \cdot \gamma'(s) \, ds \bigr| \, \leq \, c_L \, d(x,y) \, ,
\]
that is, $\Psi: (H,d) \to \R$ is Lipschitz continuous with Lipschitz constant $c_L$.
\end{rem}

\begin{proof}
Let $v=e^\Psi u$. Formally we multiply the equation by $e^\Psi v$, and then integrate with respect to $x \in H$ to get
\begin{align*}
\partial_t \| v \|_{L^2(\mu_1)}^2 \, = \, - 2 \int_H \bigl( 
&(\Delta v)^2 + 2 \, |\nabla \Psi|^2 \, v \, \Delta v + |\nabla\Psi|^4 \, v^2 - 4 \, |\nabla\Psi|^2 \, |\nabla v|^2 - (\Delta\Psi)^2 \, v^2 - \bigr. \\
\bigl. &- \, 2 \, \nabla\Psi\cdot\nabla v (\Delta\Psi) v \bigr) \, d \mu_3 - 8 \int_{H} \bigl( |\nabla' v|^2 - |\nabla' \Psi|^2 \, v^2 \bigr) \, d\mu_1 \, .
\end{align*}
Now using the Cauchy-Schwarz inequality, the first integral is bounded (up to some constant $c>0$) by
\[
- \| \Delta v \|_{L^2(\mu_3)}^2 + \| |\nabla\Psi|^2 v \|_{L^2(\mu_3)}^2 + \| \Delta \Psi \, v \|_{L^2(\mu_3)}^2 + \| \nabla\Psi \cdot \nabla v \|_{L^2(\mu_3)}^2 \, .
\]
Here we use the properties of the Lipschitz function $\Psi$ which, combined with identity (\ref{auxiliary formula}), implies that
\[
\partial_t \| v \|_{L^2(\mu_1)}^2 \, \lesssim \, - \| D_{\!x}^2 v \|_{L^2(\mu_3)}^2 + (c_L^{\,2} + c_L^{\,4}) \, \| v \|_{L^2(\mu_1)}^2 + c_L^{\,2} \, \| \nabla v \|_{L^2(\mu_2)}^2 + c_L^{\,2} \, \| v \|_{L^2}^2 \, .
\]

In order to bound $\| v \|_{L^2}$, we apply Hardy's inequality to get $\| v \|_{L^2} \lesssim \| \nabla v \|_{L^2(\mu_2)}$. Hence we arrive at
\[
\partial_t \| v \|_{L^2(\mu_1)}^2 \, \lesssim \, -(1 - c_L^{\,2} \, \eps) \, \| D_{\!x}^2 v \|_{L^2(\mu_3)}^2 + c_n (c_L^{\,2} + c_L^{\,4} + \frac{c_L^{\,2}}{\eps} \,) \, \| v \|_{L^2(\mu_1)}^2
\]
by means of
\[
\| \nabla v \|_{L^2(\mu_2)}^2 \, \leq \, c_n \, \| v \|_{L^2(\mu_1)} \| D_{\!x}^2 v \|_{L^2(\mu_3)} \, ,
\]
a weighted version of the Gagliardo-Nirenberg interpolation inequality. Taking $\eps \sim c_L^{-2}$ small, we thus have
\[
\partial_t \| v \|_{L^2(\mu_1)}^2 \, \lesssim \, - \| D_{\!x}^2 v \|_{L^2(\mu_3)}^2 + c_n (c_L^{\,2} + c_L^{\,4}) \, \| v \|_{L^2(\mu_1)}^2 \, .
\]
With
\[
F(t) \, = \, e^{-c_n (c_L^{\,2} + c_L^{\,4}) t} \| v(t) \|_{L^2(\mu_1)}^2 + \int_0^t e^{-c_n (c_L^{\,2} + c_L^{\,4}) \tau} \| D_{\!x}^2 v(\tau) \|_{L^2(\mu_3)}^2 \, d\tau \, ,
\]
we find $\partial_t F(t) \leq 0$ giving $F(t) \leq F(0) = \| e^\Psi g \|_{L^2(\mu_1)}^2$ for any $t \in [0, T]$ as desired.
\end{proof}

Lemma \ref{exponential norm decrease} enables us to prove the following result.

\begin{prop}
\label{pointwise exponential estimate}
Let $I = (0, T)$ be open, $l\in\N_0$ and $\alpha$ be any multi-index. We further suppose that $\Psi$ is as in Lemma \ref{exponential norm decrease} and $u$ is a solution of $\partial_t u + Lu=0$ on $[0,T) \times \hs$ with $u(0) = g \in L^2(\mu_1)$. Then
\[
|\partial_t^l \partial_x^\alpha u(t,x)| \, \leq \, c(n,l,\alpha) \; \frac{\delta_{l,\alpha}(\sqrt[4]{t}, x)}{|B_{\!\sqrt[4]{t}}(x)|_1^{\,\frac{1}{2}}} \; e^{c_n (c_L^{\,2} + c_L^{\,4}) t - \Psi(x)}  \| e^\Psi g \|_{L^2(\mu_1)}
\]
for all $(t,x) \in \bar{I} \setminus \{0\} \times \hs$.
\end{prop}

\begin{proof}
We argue as in the proof of Corollary \ref{pointwise estimate by initial data}, with the additional factor $1 = e^{\Psi-\Psi}$ in the norm, to get
\[
|\partial_t^l \partial_x^\alpha u(t , x)| \, \leq \, \ \frac{\delta_{l,\alpha}(\sqrt[4]{t}, x)}{|B_{\!\sqrt[4]{t}}(x)|_1^{\,\frac{1}{2}}} \sup_{z\in B_{\!\sqrt[4]{t}}(x)} e^{-\Psi(z)} \sup_{\tau \in (0,t)} \| e^{\Psi} u(\tau) \|_{L^2(\mu_1)} \, .
\] 
Now we can use Lemma \ref{exponential norm decrease} to factor in the initial value. Eventually, the Lipschitz continuity of $\Psi$ guarantees that
\[
e^{\Psi(x)-\Psi(z)} \, \leq \, e^{c_L \sqrt[4]{t}} \, \leq \, 
\begin{cases}
e^{c_L^{\,4} t} 	& \text{if} \; c_L \sqrt[4]{t} \geq 1 \\
e 			& \text{otherwise}
\end{cases}
\]
for all $z \in B_{\!\sqrt[4]{t}}(x)$, giving the estimate the appearance stated in the proposition.
\end{proof}

Note that Proposition \ref{pointwise exponential estimate} implies Corollary \ref{pointwise estimate by initial data} by setting $c_L=0$. The exponential function, however, enables us to obtain a pointwise control by rough initial data, i.e.\ initial data with finite Lipschitz norm
\[
\| g \|_{\dot{C}^{0,1}(H)} = \| \nabla g \|_{L^{\infty}(H)} \, .
\]

\begin{prop}
\label{estimate by rough initial data}
Let $I = (0,T) \subseteq \R$, $l\in\N_0$ and $\alpha$ be any multi-index with either $l \neq 0$ or $\alpha \neq 0$. If $u$ is an energy solution of $\partial_t u+Lu=0$ on $[0, T) \times \hs$ with $u(0) = g \in L^2(\mu_1)$, then we have the estimate
\[
|\partial_t^l \partial_x^\alpha u(t,x)| \, \leq \, c(n, l,\alpha) \, \delta_{l,\alpha}(\sqrt[4]{t},x) \, \sqrt[4]{t} \, (\sqrt[4]{t} + \sqrt{x_n}) \, \| g \|_{\dot{C}^{0,1}}
\]
for all $(t,x) \in  \bar{I} \setminus \{0\} \times \hs$.
\end{prop}

Using this we can control the $L^p$-norm on the cylinder bounded away from initial time $0$. To this end, we define
\[
Q_R(x) \, = \, ( \frac{R^4}{2}, R^4] \times B_R(x)
\]
for $R>0$ and $x \in \hs$.

\begin{cor}
\label{Lp estimates in the cylinder}
Let $I=(0,T)$, $j \geq 0$, $l \in \N_0$ and $\alpha$ be a multi-index with either $l \neq 0$ or $\alpha \neq 0$. If $u$ is an energy solution of $\partial_t u + Lu = 0$ on $[0,T) \times \hs$ with initial condition $u(0)=g \in L^2(\mu_1)$, then
\[
\| \nabla u \|_{L^\infty(I \times H)} + |Q_R(x)|^{-\frac{1}{p}} R^{4l+|\alpha|-1} (R + \sqrt{x_n})^{|\alpha|-2j- 1} \, \| \partial_t^l \partial_x^\alpha u \|_{L^p(Q_R(x), \mu_{jp})} \, \leq \, c \, \| g \|_{\dot{C}^{0,1}} 
\]
for all $p \in [1,\infty)$, $0 < R \leq \sqrt[4]{T}$ and $x\in\hs$.
\end{cor}

\begin{proof}
For the first part of the estimate we apply Proposition \ref{estimate by rough initial data} with $l=0$, $|\alpha|=1$. For $p < \infty$, we get
\[
\| \partial_t^l \partial_x^\alpha u \|_{L^p(Q_{R}(x),\mu_{jp})} \, \lesssim \,
\sup_{(t,y) \in Q_R(x)} \sqrt[4]{t}^{1 - 4l - |\alpha|} (\sqrt[4]{t} + \sqrt{y_n})^{2j+1-|\alpha|} \, \| g \|_{\dot{C}^{0,1}} \, |Q_R(x)|^\frac{1}{p}
\]
and the complete statement follows from $\sqrt{y_n} \lesssim R + \sqrt{x_n}$ and $\sqrt{x_n} \lesssim R + \sqrt{y_n}$ for $y\in B_{R}(x)$, as well as $\sqrt[4]{t} \sim R$ for $t \in (\frac{R^4}{2}, R^4]$.
\end{proof}

\subsection{Gaussian estimates} 
\label{Gaussian estimates}
Solutions of parabolic equations are often given by kernels which in turn can be estimated by Gaussian functions. For example, Koch and Lamm \cite{KL12} show that the biharmonic heat kernel $G(t,x,y)$ that is associated to the equation $\partial_t u + \Delta^2u=0$ has a pointwise control of the type
\[
|G(t,x,y)| \, \leq \, c \, t^{-\frac{n}{4}} e^{- \eps \, \bigl( \frac{|x-y|^4}{t} \bigr)^\frac{1}{3}} \, .
\]
The power of $t$ in front of the Gaussian factor appears in situations in which the volume of a ball is comparable to its radius - here a Euclidean setting is considered with $| B_{\!\sqrt[4]{t}}^{eu}(x)| \sim t^\frac{n}{4}$ for every $x \in \R^n$. In non-Euclidean situations, on the other hand, one has to replace this factor by an expression of the form
\[
\mu\bigl( B_{\!\sqrt[4]{t}}(x) \bigr)^{-\frac{1}{2}} \, \mu\bigl( B_{\!\sqrt[4]{t}}(y) \bigr)^{-\frac{1}{2}} \, ,
\]
where $\mu$ denotes the underlying measure and $B_{\!\sqrt[4]{t}}(\cdot)$ denotes the intrinsic ball of radius $\sqrt[4]{t}$. This illustrates that both analytic and geometric properties are combined by the kernel $G$. We now want to derive a Gaussian estimate for the Green function in terms of the intrinsic metric $d$ and the measure $\mu_1$. The approach presented here follows an idea of Fabes and Stroock \cite{FS86}.

\begin{thm}
\label{Gaussian estimate}
For $T>0$ let $I=(0,T)$ be an open interval in $\R$, $l \in \N_0$ and $\alpha$ be any multi-index. Then there exists a Green function $G: I \times \hs \times I \times \hs \to \R$ with $G(t,x,s,y)=0$ for $t < s \in [0, T)$, and
\[
\partial_t^l \partial_x^\alpha \, u(t,x) \; = \; \int_H \partial_t^l \partial_x^\alpha G(t,x,s,y) u(s,y) \, dy
\]
for all $0 \leq s < t \leq T$, $x\in\hs$ and any energy solution $u$ of $\partial_t u + Lu=0$ on $[0, T) \times \hs$ with initial condition $u(0)=g$. In particular, there exist positive constants $c=c(n, l, \alpha)$ and $c_{n}=c(n)$ such that
\begin{equation}
\label{ge}
\tag{ge}
|\partial_t^l \partial_x^\alpha G(t, x, s, y)| \, \leq \, c \, \delta_{l, \alpha}(\sqrt[4]{t-s},x) |B_{\!\sqrt[4]{t-s}}(x)|_1^{-\frac{1}{2}} |B_{\!\sqrt[4]{t-s}}(y)|_1^{-\frac{1}{2}} y_n \, e^{- c_n^{-1} \bigl( \frac{d(x,y)^4}{t-s} \bigr)^\frac{1}{3}}
\end{equation}
for almost every $x \neq y \in \hs$.
\end{thm}

\begin{cor}
\label{IVP for other data than L2}
The Gaussian estimate allows us to solve the initial value problem also for other data than those in $L^2(H, \mu_1)$.
\end{cor}

\begin{proof}[Sketch of proof]
Given an initial datum $g$ in either $L^1(\mu_1)$ or $\dot{C}^{0,1}$, one can truncate $g$ to become a function in $L^2(\mu_1)$. Using the representation by the Green function one obtains a solution and the exponential decay ensures convergence of the truncated solution.
\end{proof}

\begin{rem}
\label{exchange of centers in ge}
We can substitute $x$ and $y$ in $|B_R(x)|_1^{-\frac{1}{2}} \, |B_R(y)|_1^{-\frac{1}{2}}$ by the respective other center, paying for it with the factor
\[
\Bigl( 1 + \frac{d(x,y)}{R} \Bigr)^{n+1} \, ,
\]
see (\ref{ball measure}). In the same manner we can replace $\delta_{l,\alpha}(R,x)$ by $\delta_{l,\alpha}(R,y)$. For instance, this shows that the estimate 
\[
|\partial_t^l \partial_x^\alpha G(t,x,s,y)| \, \leq \, c \; \frac{\delta_{l, \alpha}(\sqrt[4]{t-s},y) \, y_n}{|B_{\!\sqrt[4]{t-s}}(x)|_1} \; e^{- c_n^{-1} \bigl( \frac{d(x,y)^4}{t-s} \bigr)^\frac{1}{3}}
\]
and (\ref{ge}) are comparable up to changing the constants $c, c_n > 0$. From now on, we always choose a suitable combination of $x$ and $y$ in the factor $\delta_{l, \alpha}(R,\cdot) |B_R(\cdot)|_1^{-1}$ and refer to an estimate of this type as ``Gaussian estimate''.
\end{rem}

We conclude this section with three estimates for the Green function $G$: The first one generalizes the Gaussian estimate to derivatives in the $s$- and $y$-variable. The second one rephrases it in a more convenient form, but limited to $(s,y)$ outside of a certain cylinder $Q$. In the last result we use (\ref{ge}) to show that, in a certain range of $q \geq 1$, $G$ and its weighted derivatives (leaving temporal derivatives aside) are in the space $L^q$, where the integral is taken with respect to $\mathcal{L}^{n+1}$.

\begin{lemma}
\label{Gaussian estimate involving s- and y-derivatives}
Let $I=(0,T)$, $l,m \in \N_0$ and $\alpha, \beta$ be multi-indices. If $G$ is the Green function associated to the homogeneous initial value problem, then there exist $\eps = \eps_n$ and $c=c(n,l,m,\alpha,\beta)$ such that
\[
|\partial_s^m \partial_y^\beta \bigl( y_n^{-1} \partial_t^l \partial_x^\alpha G(t, x, s, y) \bigr)| \, \leq \, c \; \frac{\delta_{l+m, \alpha+\beta}(\sqrt[4]{t-s},x)}{|B_{\!\sqrt[4]{t-s}}(x)|_1^{\,\frac{1}{2}} \, |B_{\!\sqrt[4]{t-s}}(y)|_1^{\,\frac{1}{2}}} \; e^{- \eps \bigl( \frac{d(x,y)^4}{t-s} \bigr)^\frac{1}{3}}
\]
for any $t > s \in I$ and $x, y \in \hs$.
\end{lemma}

\begin{proof}
First, one may check that $y_n^{-1} G(t,x,s,y)$ is a solution with respect to $(s,y)$ on $I \times \hs$ in the sense of Definition \ref{energy solution}, and hence the same holds true for $y_n^{-1} \partial_t^l \partial_x^\alpha G(t,x,s,y)$. One has to use the symmetry property 
\[
G(t,x,s,y) \, = \, \frac{y_n}{x_n} \; G(s,y,t,x) \, .
\]
Thus Proposition \ref{pointwise estimate} applied to $\xi_n^{-1} \partial_t^l \partial_x^\alpha G(t,x,\tau,\xi)$ in $Q_R(s, y)$, with a suitable choice of $R \sim \sqrt[4]{t-s}$, implies
\[
|\partial_s^m \partial_y^\beta \bigl( y_n^{-1} \partial_t^l \partial_x^\alpha G(t, x, s, y) \bigr)| \, \lesssim \, \delta_{m, \beta}(R,y) \, \| \xi_n^{-1} \partial_t^l \partial_x^\alpha G(t, x, \cdot, \cdot) \|_{L^{\infty}(Q_R(s, y))} \, .
\]
We can now use the Gaussian estimate proposed in Theorem \ref{Gaussian estimate} with respect to $(t,x)$ to bound this by
\[
\delta_{m, \beta}(R,x) \, \delta_{l, \alpha}(R, y) \, |B_R(x)|_1^{-\frac{1}{2}} \, |B_R(y)|_1^{-\frac{1}{2}} \, \| e^{- \eps \bigl( \frac{d(x \, , \, \cdot)}{R} \bigr)^\frac{4}{3}} \|_{L^{\infty}(Q_R(s, y))} \, .
\]
Finally, we use the triangle inequality to show that $d(x ,\xi)^\frac{4}{3} \geq \frac{1}{2} \, d(x,y)^\frac{4}{3} - d(\xi,y)^\frac{4}{3}$. Hence the assertion follows since $d(\xi,y) < R$ for $\xi \in B_R(y)$.
\end{proof}

\begin{lemma}
\label{Gaussian estimate outside of a cylinder}
Let $G$ be the Green function to $\partial_t u+Lu=0$ on $(0,1) \times \hs$ and $(t,x) \in (\frac{1}{2},1] \times \hs$. Then, for every $j \geq 0$, $l\in\N_0$ and any multi-index $\alpha \in {\N_0}^{\!n}$ satisfying $|\alpha| \geq 2j$, we have the estimate
\[
x_n^{\,j} \, |\partial_t^l \partial_x^\alpha G(t,x,s,y)| \, \leq \, c \, (1+\sqrt{y_n})^{2j-|\alpha|} \, |B_1(y)|^{-1} \, e^{- \frac{d(x,y)}{4 \, c_n}}
\]
for almost all $(s,y) \in \bigl( (0,t] \times \hs \bigr) \setminus \bigl( (\frac{1}{4},t] \times B_1(x) \bigr)$. Here, $c_n$ is the constant from the Gaussian estimate (\ref{ge}) and $c$ depends on $n,j,l$ and $\alpha$.
\end{lemma}

\begin{proof}
With $Q = (\frac{1}{4}, t] \times B_1(x)$ for $(t,x) \in [\frac{1}{2},1] \times \hs$, we decompose
\[
\bigl( (0,t] \times \hs \bigr) \setminus Q \, = \, (0,\frac{1}{4}] \times B_1(x) \dotcup (0,\frac{1}{4}] \times B_1(x)^c \dotcup (\frac{1}{4},t] \times B_1(x)^c \, = \, \bigcup_{i=1}^3 M_i \, .
\]
A straightforward computation now shows that, for almost all $(s,y) \in M_1 \dotcup M_2 \dotcup M_3$, we have the estimate
\[
\sqrt[4]{t-s}^{-2n-4l-2|\alpha|+2j} \, e^{-(2 c_n)^{-1} \, \bigl( \frac{d(x,y)^4}{t-s} \bigr)^\frac{1}{3}} \, \leq \, c(n,j,l,\alpha) \, e^{- \frac{d(x,y)}{4 c_n}} \, ,
\]
and the assertion follows from Theorem \ref{Gaussian estimate} and the doubling property (\ref{doubling condition}), or rather (\ref{ball measure}).
\end{proof}

\begin{lemma}
\label{certain derivatives of Phi are in Lq}
Let $G$ be as in Lemma \ref{Gaussian estimate outside of a cylinder}, $j \geq 0$ and $\alpha$ be a multi-index satisfying $2j \leq |\alpha| < j+2$. Then 
\[
\| x_n^{\,j} \partial_x^{\alpha} G(t,x,\cdot,\cdot) \|_{L^q((0,t) \times H)} \, \leq \, c(n,j,\alpha,q) \, (1+\sqrt{x_n})^{2j-|\alpha|} \, |B_1(x)|^{\frac{1}{q}-1}
\] 
for all $t \in (0,1]$, almost all $x\in\hs$ and for any $1 \leq q < \frac{n+2}{n-j+|\alpha|}\,$.
\end{lemma}

\begin{proof}
Let $t > s$ and $x \in \hs$. With $A_i(x) = B_{i \sqrt[4]{t-s}}(x) \setminus B_{(i-1) \sqrt[4]{t-s}}(x)$, we decompose the half space into the annuli
\[
H \, = \, \underset{i\in\N}{\bigcup} A_i(x) \, .
\]
Applying the Gaussian estimate (\ref{ge}), followed by repeated application of the formula (\ref{ball measure}), we obtain
\[
\int_H x_n^{\,j q} \, |\partial_x^\alpha G(t,x,s,y)|^q \, dy
\]
\[
\lesssim \, (1+\sqrt{x_n})^{(2j-|\alpha|) q} \, |B_1(x)|^{1-q} \, \sqrt[4]{t-s}^{2n-2q(n-j+|\alpha|)} \sum_{i\in\N} i^{2n+2q} \, e^{- q c_n^{-1} (i-1)^\frac{4}{3}}
\]
if $|\alpha| \geq 2j$. We subsume the convergent series into the constant and then integrate in $s\in(0,t)$ to get
\[
\int_0^t \| x_n^{\,j} \partial_x^\alpha G(t,x,s,\cdot) \|_{L^q}^q \, ds \, \lesssim \, (1+\sqrt{x_n})^{(2j-|\alpha|) q} \, |B_1(x)|^{1-q} \; \frac{\tau^{\frac{n+2}{2} - \frac{q}{2} ( n-j+|\alpha| )}}{n+2-q (n-j+|\alpha|)} \; \bigg\vert_{\tau=0}^{\tau=t}
\]
which is bounded above by a constant depending on $n,j,\alpha$ and $q$, if $q < \frac{n+2}{n-j+|\alpha|}$. But since also $q \geq 1$, this condition is satisfied as long as $-j + |\alpha| < 2$.
\end{proof}

\begin{rem}
\label{integrability of G wrt t and x}
Following the same line of argument we can show that, for any $2j \leq |\alpha| < j+2$, we have
\[
\bigl( \int_s^1 \| \partial_x^{\alpha} G(t,\cdot,s,y) \|_{L^q(\mu_{jq} )}^q \, dt \bigr)^\frac{1}{q} \, \leq \, c(n,j,\alpha,q) \, (1+\sqrt{y_n})^{2j-|\alpha|} \, |B_1(y)|^{\frac{1}{q}-1}
\]
for all $s \in (0,1]$, almost all $y\in\hs$ and for any $1 \leq q < \frac{n+2}{n-j+|\alpha|}$.
\end{rem}

\subsection{Kernel estimates} 
\label{kernel estimates}
We now consider again the inhomogeneous equation $\partial_t u + Lu=f$ with zero Cauchy data and note, using Duhamel's principle, that any solutions can be written in integral form
\[
u(t,x) \, = \, \int_0^t \int_H G(t,x,s,y) f(s,y) \, dy \, ds \, .
\]
This enables us to treat $f \mapsto x_n^{\,j} \partial_t^l \partial_x^\alpha u$ as integral kernel operator on a homogeneous-type space.

\begin{lemma}
\label{Lp boundedness for Schur operators}
Let $1 \leq p \leq \infty$ and $u$ be a solution of the inhomogeneous equation on $[0,1) \times \hs$ with $u(0)=0$. Then
\[
\| x_n^{\,j} \partial_x^\alpha u \|_{L^p((0,1) \times H)} \, \leq \, c(n,j,\alpha) \, \| f \|_{L^p((0,1) \times H)}
\]
for any $j \geq 0$ and any multi-index $\alpha$ with $2j \leq |\alpha| < j+2$, and especially for $j=\frac{|\alpha|}{2}$ if $|\alpha| < 4$.
\end{lemma}

\begin{proof}
Since $2j \leq |\alpha| < j+2$, we can apply Lemma \ref{certain derivatives of Phi are in Lq} and Remark \ref{integrability of G wrt t and x} to get the following kernel estimates:
\[
\| x_n^{\,j} \partial_x^\alpha G(t,x,\cdot,\cdot) \|_{L^1((0,t) \times H)} \, \lesssim \, (1+\sqrt{x_n})^{2j-|\alpha|} \, \leq \, 1
\]
and
\[
\int_s^1 \| \partial_x^\alpha G(t,\cdot,s,y) \|_{L^1(H, \mu_j)} \, dt \, \lesssim \, (1+\sqrt{y_n})^{2j-|\alpha|} \, \leq \, 1 \, .
\]
This implies that $f \mapsto x_n^{\,j} \partial_x^\alpha u$ is a bounded operator from $L^p\bigl( (0,1) \times H \bigr)$ into itself.\footnote{This statement, due to Schur, is a basic result for  integral kernel operators. A proof is provided in \cite{F84}.}
\end{proof}

Let $V(t,x,s,y)$ be the volume of the ``smallest'' ball centered at $(t,x)$ that contains $(s,y)$. As the volume function $V$ is essentially symmetric, i.e.\ $V(t,x,s,y) \sim V(s,y,t,x)$, it is equivalent to say $V$ is given by
\[
V(t,x,s,y) \, = \, |B_{d_0}(t,x)|_1 + |B_{d_0}(s,y)|_1 \, ,
\]
where $d_0 = d^{(t)}\bigl( (t,x), (s,y) \bigr) = \sqrt[4]{|t-s| + d(x,y)^4}$. Moreover, suppressing all the arguments, we define
\[
D \, = \, \frac{d^{(t)}\bigl((t,x),(\bar{t},\bar{x})\bigr) + d^{(t)}\bigl((s,y),(\bar{s},\bar{y})\bigr)}{d^{(t)}\bigl((t,x),(s,y)\bigr) + d^{(t)}\bigl((\bar{t},\bar{x}),(\bar{s},\bar{y})\bigr)} \, .
\]

\begin{prop}
\label{CZ-kernel estimate}
Let $G$ be the Green function and $K$ be given by any of the following expressions:
\[
y_n^{-1} \partial_t G(t,x,s,y) \, , \; y_n^{-1} D_{\!x}^2 G(t,x,s,y) \, , \; y_n^{-1} x_n D_{\!x}^3 G(t,x,s,y) \, , \quad \text{or} \quad y_n^{-1} x_n^{\,2} D_{\!x}^4 G(t,x,s,y) \, .
\]
Then there exists $C=C(n)>0$ such that $|K(t,x,s,y)| \leq C \, V(t,x,s,y)^{-1}$, and if in addition $D \leq \frac{1}{6}$,
\[
|K(t,x,s,y) - K(\bar{t},\bar{x},\bar{s},\bar{y})| \, \leq \, C \; \frac{D}{V(t,x,s,y)} \; .
\]
\end{prop}

The kernel estimates in Proposition \ref{CZ-kernel estimate} and the energy estimate in Proposition \ref{energy estimate} provide all that is needed to apply the theory for singular integral operators. For $j,l$ and $\alpha$ admissible in the sense that
\begin{equation}
\label{CZ exponents}
(j,l,\alpha) \, \in \, \mathcal{CZ} \, = \, \bigl\{ (j, l, \alpha) \in [0,\infty) \times \N_0 \times \N_0^{\,n} \mid j=2l + |\alpha| - 2 \; \text{and} \; 2j \leq |\alpha| \bigr\} \, ,
\end{equation}
we obtain that the operator $f \mapsto x_n^{\,j} \partial_t^l \partial_x^\alpha u$ is a Calder\'{o}n-Zygmund operator on a homogeneous-type metric space, and therefore it is bounded on $L^p\bigl( I; L^p(\mu_1) \bigr)$ for any $p \in (1,\infty)$.

\begin{cor}
\label{CZ operators}
Let $I=(0,T)$, $u$ satisfy the linear equation on $[0,T) \times \hs$ to $f \in L^2\bigl( I; L^2(\mu_1) \bigr)$ with $u(0)=0$ and
\[
T: L^2\bigl( I; L^2(H, \mu_{1}) \bigr) \ni f \mapsto x_n^{\,j} \partial_t^l \partial_x^\alpha u \in L^2\bigl( I; L^2(H, \mu_{1}) \bigr) \, .
\]
Then $T$ is a Calder\'{o}n-Zygmund operator on $\bigl( I \times \hs, d_0 \,, \mathcal{L} \times \mu_1  \bigr)$ if and only if $(j, l, \alpha) \in \mathcal{CZ}$, i.e.\ if $T$ assigns to $f$ either
\[
\partial_t u \, , \; D_{\!x}^2 u \, , \; x_n D_{\!x}^3 u \quad \text{or} \quad x_n^{\,2} D_{\!x}^4 u \, .
\]
\end{cor}

Now, the theory of Muckenhoupt weights provides us with a tool to dispense with the weight. Such a weight on $(\hs,d,\mu_1)$ is a nonnegative and locally $\mu_1$-integrable function $\omega: \hs \to \R$ such that
\[
\sup_B \, |B|_1^{-1} \int_B \omega \, d\mu_1 \bigl[ |B|_1^{-1} \int_B \omega^{-\frac{1}{p-1}} \, d\mu_1 \bigr]^{p-1} \, \leq \, c(p,\omega) \, < \, \infty \, ,
\]
where the supremum is taken with respect to all $d$-balls $B$. We write $\omega \in A_p(\mu_1)$ and the best $A_p$ constant is denoted by $[\omega]_{A_p}$. The main result of this theory is the following: Let $p \in (1,\infty)$, $\omega \in A_p(\mu_1)$ and $T$ a Calder\'{o}n-Zygmund operator. Then there exists a positive constant $c=c(b,p,\omega)$ such that
\begin{equation}
\label{Muckenhoupt's Lp-estimate}
\| Tf \|_{L^p(\omega)} \, \leq \, c \, \| f \|_{L^p(\omega)}
\end{equation}
for all $f\in L^p(\omega)$. Here, $b \geq 1$ is the doubling constant from (\ref{doubling condition}). For further reading on Muckenhoupt theory on spaces of homogeneous type we refer to \cite{K04} and the references therein.

\begin{prop}
\label{weighted Lp-estimate}
Suppose $I=(0,T)$ is open, $f \in L^2\bigl( I; L^2(\mu_1) \bigr)$ and $p \in (1,\infty)$. Assume further that $u$ is an energy solution of $\partial_t u + Lu=f$ on $[0,T) \times \hs$ with $u(0)=0$ and $(j,l,\alpha) \in \mathcal{CZ}$. Then we have
\[
\| x_n^{\sigma+j} \partial_t^l \partial_x^\alpha u \|_{L^p(I \times H)} \, \lesssim \, \| x_n^{\sigma} f \|_{L^p(I \times H)}
\]
for all $- \frac{1}{p} < \sigma < 2 - \frac{1}{p}$. In particular, this holds true for $\sigma \in [0,1]$.
\end{prop}

\begin{proof}
A straightforward calculation, using the definition of the Muckenhoupt class $A_p(\mu_1)$, shows that $x_n^{\,\sigma p-1} \in A_p(\mu_1)$ if and only if $-1 < \sigma p < 2p-1$. But then the statement follows from Corollary \ref{CZ operators} in conjunction with estimate (\ref{Muckenhoupt's Lp-estimate}).
\end{proof}

\subsection{The spaces $X_p$ and $Y_p$} 
\label{function spaces}
The next step consists in defining appropriate function spaces. Here and in the rest of this paper we assume that the initial value $g$ belongs to the homogeneous Lipschitz space $\dot{C}^{0,1}(H)$, that is, $\| g \|_{\dot{C}^{0,1}} = \| \nabla g \|_{L^{\infty}} < \infty$. By Corollary \ref{Lp estimates in the cylinder} this is a natural bound on the solution of the homogeneous initial value problem and hence motivates the definition of a new norm, denoted $X_p$: For $p \in [1,\infty)$, we refer to $X_p$ as the function space of all functions with finite norm
\[
\| u \|_{X_p} \, = \, \| \nabla u \|_{L^{\infty}(I \times H)} + \| u \|_{X_p^1} \, ,
\]
where
\[
\| u \|_{X_p^1} \, = \, \sup_{{R^4 \in (0,T)}\atop{x\in H}} \, |Q_R(x)|^{-\frac{1}{p}} \sum_{(j, l, \alpha) \in \mathcal{CZ}} R^{4l+|\alpha|-1} \, (R+\sqrt{x_n})^{|\alpha|-2j-1} \, \| \partial_t^l \partial_x^\alpha u \|_{L^p(Q_R(x), \mu_{jp})} \, .
\]
By $\bar{B}_\eps^X = \{ u\in X_p \mid \| u \|_{X_p} \leq \eps \}$ we denote the closed $\eps$-ball in $X_p$. \\[1em]
\textbf{Remark:} In order to have any chance to bound this by the inhomogeneity $f$, the sum has to be taken over $(j,l,\alpha) \in \mathcal{CZ}$ (see Proposition \ref{weighted Lp-estimate}), rather than all possible combinations of $j,l$ and $\alpha$ as proposed in Corollary \ref{Lp estimates in the cylinder}. \\[1em]
Finally, we define the Banach space $Y_p$ based on time-space cylinders by
\[
Y_p \; = \; \bigl\{ f \mid \| f \|_{Y_p} = \sup_{{R^4 \in (0,T)}\atop{x\in H}} |Q_R(x)|^{-\frac{1}{p}} \, R^3 \, (R+\sqrt{x_n})^{-1} \, \| f \|_{L^p(Q_R(x))} < \infty \bigr\} \, .
\]

We recall the invariance of the linear equation under the scaling defined in (\ref{linear scaling}) and note that the solution and the initial datum exhibit the same scaling behavior in their respective norms. More precisely,
\[
\lambda \, \| u \|_{X_p} \, \sim \, \| u \circ T_\lambda \|_{X_p} \quad \text{and} \quad 
\lambda \, \| g \|_{\dot{C}^{0,1}} \, = \, \| g(\lambda \, \cdot) \|_{\dot{C}^{0,1}} \, .
\]
As opposed to this, the scaling of the $Y_p$-norm is characterized by the estimate
\begin{equation}
\label{scaling of the Y-norm}
\| f \circ T_\lambda \|_{Y_p} \, \leq \, c(n,p) \, \lambda^{-1} \, \| f \|_{Y_p} \, .
\end{equation}
Our goal now is to show that
\[
\| u \circ T_\lambda \|_{X_p} \, \lesssim \, \lambda^2 \, \| f \circ T_\lambda \|_{Y_p}
\]
for solutions of the linear problem with $u(0)=0$. Then this implies the following crucial result.

\begin{prop}
\label{X-norm vs Y-norm}
Let $I = (0,T)$ be open, $f \in Y_p$ for $p \in (n+2,\infty)$ and $g \in \dot{C}^{0,1}(H)$. Then we have
\[
\| u \|_{X_p} \, \leq \, c(n,p) \, \bigl( \| f \|_{Y_p} + \| g \|_{\dot{C}^{0,1}} \bigr)
\]
for any energy solution $u$ of $\partial_t u + Lu = f$ on $[0,T) \times \hs$ satisfying the initial condition $u(0)=g$.
\end{prop}


\section{Main results} 
\label{main results}
In this section we finally turn to the nonlinear equation for the perturbation of the steady state given by
\[
\tag{pe}
\partial_t u + Lu \, = \, f_0[u] + x_n \, f_1[u] + x_n^{\,2} \, f_2[u] \, , \quad u(0) \, = \, g \, ,
\]
see Lemma \ref{nonlinearity}. Using Duhamel's principle, we can rewrite this equation on $[0,T) \times \hs$ in integral form
\begin{equation}
\label{solution via Duhamel}
F_g(u) \, = \, F(g,u) \, = \, Sg + \Psi f[u] \, ,
\end{equation}
where 
\[
S g(t,x) \, = \, \int_H G(t,x,0,y) g(y) \, dy \quad \text{and} \quad \Psi f[u](t,x) \, = \, \int_0^t \int_H G(t,x,s,y) f[u](s,y) \, dy ds \, .
\] 
In the first part of this section we can use the linear estimates obtained in section \ref{linear model case} to implement a fixed point argument in the function space $X_p$. In order to do so, one needs to impose additional requirements concerning the nonlinearity $f[u]$, the maximal time of existence $T$ or the initial data $g$, at least one of which needs to be small. Much in the spirit of Koch and Lamm \cite{KL12} we reach global existence and uniqueness for (\ref{perturbation equation}) from small Lipschitz data.

The goal for the rest of this section is to use the unique solution to generate a solution $h$ on its positivity set $P(h)$.

\subsection{Nonlinear estimates} 
\label{nonlinear estimates}

Under a smallness assumption on the solution we obtain the estimate from Proposition \ref{X-norm vs Y-norm} in the opposite direction, i.e.\ an estimate of the $Y_p$-norm by the $X_p$-norm.

\begin{lemma}
\label{Y-norm vs X-norm}
Let $I=(0,T)$ be an open interval, $p \in [1,\infty)$ and $\eps<\frac{1}{2}$. We further suppose that $f: X_p \to Y_p$ is defined as in Lemma \ref{nonlinearity}. Then the operator $f: \bar{B}_\eps^X \to Y_p$ is analytic and we have the estimates
\[
\| f[u] \|_{Y_p} \, \leq \, c \, \| u \|_{X_p}^2
\]
for all $u \in \bar{B}_\eps^X$ and
\[
\| f[u_1] - f[u_2] \|_{Y_p} \, \leq \, c  \, \bigl( \| u_1 \|_{X_p} + \| u_2 \|_{X_p} \bigr) \, \| u_1 - u_2 \|_{X_p}
\]
for all $u_1, u_2 \in \bar{B}_\eps^X$, where the constant $c$ depends only on $n$ and $p$.
\end{lemma}

Next we combine the results from Propositoin \ref{X-norm vs Y-norm} and Lemma \ref{Y-norm vs X-norm} to prove the following theorem.

\begin{thm}
\label{well-posedness of the nonlinear problem}
Let $T>0$ and $p \in (n+2,\infty)$. Then there exist $\eps, c > 0$ such that for every $g \in \dot{C}^{\, 0,1}(H)$ with $\| g \|_{\dot{C}^{0,1}}<\eps$ there exists a solution $u^* \in X_p$ of the perturbation equation (\ref{perturbation equation}) for which
\[
\| u^* \|_{X_p} \; \leq \; c(n,p) \, \| g \|_{\dot{C}^{\,0,1}}
\]
holds. Moreover, this solution is unique in $\bar{B}_{c\eps}^X$.
\end{thm}

\begin{proof}
For every $g \in \dot{C}^{0,1}(H)$, by (\ref{solution via Duhamel}) there is an operator $F_g: X_p \ni u \mapsto \tilde{u} \in X_p$ defined such that
\[
\partial_t \tilde{u} + L\tilde{u} \, = \, f[u] \, , \quad \tilde{u}(0) \, = \, g \, .
\]
Due to Proposition \ref{X-norm vs Y-norm} and Lemma \ref{Y-norm vs X-norm} there exist $\delta, \eps > 0$ such that, for all $g \in \dot{C}^{0,1}(H)$ with $\| g \|_{\dot{C}^{0,1}} < \eps$, the mapping $F_g: \bar{B}_\delta^X \to \bar{B}_\delta^X$ has a unique fixed point $u^* \in \bar{B}_\delta^X$ that depends on the initial value in a Lipschitz continuous way. Indeed, using the second inequality in Lemma \ref{Y-norm vs X-norm}, we arrive at
\[
\| F_g(u_1) - F_g(u_2) \|_{X_p} \, \leq \, c_L \, \| u_1 - u_2 \|_{X_p}
\]
for a Lipschitz constant $c_L \in (0,1)$ provided that $u_1, u_2 \in X_p$ are chosen sufficiently small. Since $F_g(u^*)=u^*$, this turns out to be the unique solution of (\ref{perturbation equation}) we have been looking for.
\end{proof}

In addition, we get that the unique solution obtained in Theorem \ref{well-posedness of the nonlinear problem} is analytic in temporal and all tangential directions. Analyticity in vertical direction ($x_n$-direction) is still an open problem.

\begin{prop}
\label{analyticity in t and x'}
Let $u^* \in \bar{B}_{c\eps}^X$ be the unique solution of (\ref{perturbation equation}). This solution depends analytically on the initial data $g \in \dot{C}^{0,1}(H)$. Moreover, $u^*$ is analytic in temporal and all tangential directions, and there exists a number $R>0$ such that for any $l \in \N_0$ and for any multi-index $\alpha' \in \N_0^{\,n-1}$ the estimate
\[
\sup_{t\in I} \, \sup_{x\in H} \, | t^{l+\frac{1}{2} |\alpha'|} \, \partial_t^l \partial_x^{\alpha'} \nabla_{\!x} u^*(t,x)| \, \leq \, c \, R^{-l-|\alpha'|} \, l! \, \alpha'! \, \| g \|_{\dot{C}^{0,1}}
\]
holds with a constant $c>0$ depending only on $n$ and $R$.
\end{prop}

As our equation is non-degenerate in $t$ and $x'$, the same arguments as used for the related equation
\[
\partial_t u + \Delta^2 u \, = \, f_0[u] + \nabla f_1[u] + D_{\!x}^2 f_2[u]
\]
apply. For full details we refer to the paper \cite{KL12} by Koch and Lamm.

\subsection{Proof of Theorem \ref{uniqueness}} 
\label{uniqueness, stability and regularity}
Before we can prove the main result of this paper, we show that the change of coordinates $(t,x) \mapsto (s,y)$ introduced in section \ref{local coords} is a quasi-isometry. 

\begin{lemma} 
\label{almost isometry}
Let $\phi: x \mapsto \bigl(x', v(x) \bigr)$ with $v: \R^n \to \R$ satisfying $|\nabla_{\!x}v-e_n| < \eps$ for an $\eps < 1$. Then we have
\[
(1-\eps) \, |x-\bar{x}| \, < \, |\phi(x) - \phi(\bar{x})| \, < \, (1-\eps)^{-1} \, |x-\bar{x}|
\]
for all $x, \bar{x} \in \R^n$.
\end{lemma}

\begin{proof}
We assume that $x_n > \bar{x}_n$ without loss of generality. Then, applying the mean value theorem in the vertical direction and invoking the assumption on $v$, there exists a number $\bar{x}_n < z < x_n$ such that
\[
|v(x)-v(\bar{x})| \, = \, |\partial_{x_n} v(z)(x_n-\bar{x}_n)| \, < \, \eps \, (x_n-\bar{x}_n) \, .
\]
This implies
\[
|\phi(x) - \phi(\bar{x}) - (x-\bar{x})| \, < \, \frac{\eps}{1-\eps} \; \min \bigl\{ |\phi(x)-\phi(\bar{x})| \, , |x-\bar{x}| \bigr\}
\]
and hence the assertion.
\end{proof}

\begin{proof}[Proof of Theorem \ref{uniqueness}]
Assume that there exist $\delta \in (0,1)$ and $C>1$ such that $v: I \times \hs \to \R$ satisfies
\begin{equation}
\label{perturbed stationary solution bounded from above and below}
\tag{$\ast$}
\delta \, \leq \, |\nabla_{\!x} v(t,x)| \, \leq \, C \, .
\end{equation}
Then, Lemma \ref{almost isometry} allows us to reparametrize the graph of $v=x_n+u$ globally via $\phi$, reversing the local transformation applied in section \ref{transformation}. Now using $\tilde{h}=x_n$ as the new independent variable we obtain
\[
\nabla_{\!y} \tilde{h} \, = \, - v_n^{-1}
	\begin{pmatrix}
	\nabla'_{\!x} v \\
	-1 
	\end{pmatrix}
\]
and thus
\[
0 \, < \, \frac{\delta+1}{2C} \, \leq \, |\nabla_{\!y} \tilde{h}| \, \leq \, \frac{C+1}{\delta} \, < \, \infty \, .
\]

We will prove the theorem in two steps. First, we show that a solution of (\ref{perturbation equation}) yields a weak solution of the thin-film equation in the sense of definition (\ref{weak solution of tfe}). Second, we prove uniqueness of this solution by imposing additional conditions on $h$ in terms of the transformed cylinders $Q_R(x)$. \\

\textit{Existence:} Let $\varphi \in C_c^{\infty}(I \times \R^n)$ be an arbitrary test function. Putting $\tilde{h}=\sqrt{h}$, we show that for all $s \in I$,
\[
\int_{\R^n} \partial_s \tilde{h}^2 \varphi \, dy \, = \, \int_{\R^n} \tilde{h}^2 \nabla_{\!y}\Delta_y \tilde{h}^2 \cdot \nabla_{\!y}\varphi \, dy \, .
\]
Under a change of coordinates $(s,y) \mapsto (t,x)$ as in section \ref{transformation}, the integral on the left hand side transforms to
\[
-2 \int_H x_n \; \frac{\partial_t u}{v_n} \; \varphi \; \frac{\partial y_n}{\partial x_n} \; dx \, = \, -2 \int_H \partial_t u \varphi \, d\mu_1 \, ,
\]
where $v_n = \partial_{x_n} v = 1 + \partial_{x_n} u$. For the second integral, we proceed as in the proof of Lemma \ref{nonlinearity} to get
\[
2 \int_H \bigl[
\begin{pmatrix}
v_n^{-1} \, \nabla'_{\!x}\\
\partial_{x_n}
\end{pmatrix}
(x_n^{\,3} \, \Delta_x u) + 2 \, x_n^{\,2} \,
\begin{pmatrix}
v_n^{-1} \, \nabla'_{\!x} \\
\partial_{x_n}
\end{pmatrix}
\partial_{x_n} u - 2 \, x_n^{\,2} \, e_n \, \Delta_x u + R(u) \bigr] \, \nabla_{\!y}\varphi \, v_n \, dx \, .
\]
Next we employ the $\star$-notation, as introduced in section \ref{linearization}, to rewrite the remainder in the form
\[
R(u) \, = \, x_n^{\,2} \, \tilde{f}_2(\nabla_{\!x}u) \star \nabla_{\!x}u \star D_{\!x}^2u + x_n^{\,3} \, \bigl( \tilde{f}_3^1(\nabla_{\!x}u) \star \nabla_{\!x}u \star D_{\!x}^3u + \tilde{f}_3^2(\nabla_{\!x}u) \star P_2(D_{\!x}^2u) \bigr) \, .
\]
Again, the functions $\tilde{f}_2$ and $\tilde{f}_3^k$ contain factors of the form $v_n^{-m}$ for some $2 \leq m \leq 5$. Integration by parts leads to
\[
2 \int_H x_n \, (\partial_t u + L u - f[u]) \, \varphi \, dx \, = \, 0
\]
and $2 x_n \varphi$ is an admissible test function. To see this, we use the formulas from section \ref{local coords} to calculate
\[
\nabla_{\!y} \bigl[
\begin{pmatrix}
\nabla'_{\!x}\\
v_n \, \partial_{x_n}
\end{pmatrix}
(x_n^{\,3} \, \Delta_x u) + 2 \, x_n^{\,2} \,
\begin{pmatrix}
\nabla'_{\!x} \\
v_n \, \partial_{x_n}
\end{pmatrix}
\partial_{x_n} u - 2 \, x_n^{\,2} \, v_n \, e_n \, \Delta_x u + v_n \, R(u) \bigr] 
\]
\[
= \, \Delta_x (x_n^{\,3} \, \Delta_x u) - 4 \, x_n \, \Delta'_x u - f[u] \, .
\]
Reversing the transformation yields the existence of a solution $h$ of  $\partial_s h + \nabla_{\!y} \cdot (h \, \nabla_{\!y}\Delta_y h)=0$ on $P(h)$. Finally, extending $h$ by $0$ outside of $spt \, h$ and applying integration by parts, we conclude that
\[
\int_I \int_{\R^n} h \partial_s \varphi + h \, \nabla\Delta h \cdot \nabla\varphi \, dy ds \, = \, - \int_I \int_{\R^n} \bigl( \partial_s h + \nabla\!\cdot\! (h \, \nabla\Delta h) \bigr) \, \varphi \, dy ds \, = \, 0 \, .
\]
The boundary terms vanish since $h$ vanishes on $\partial P(h)$. \\

\textit{Uniqueness:} Given $g_v$ with $|\nabla_{\!x} g_v - e_n| < \eps$, then by Theorem \ref{well-posedness of the nonlinear problem} there exists a unique solution $u^*$ of (\ref{perturbation equation}) satisfying the initial condition $u^*(0)=g_v$. Moreover, with $v^*=x_n + u^*$ as above, we have
\[
\| v^*-x_n \|_{X_p} \, \lesssim \, \eps \, .
\]
In particular, $|\nabla_{\!x}v^* - e_n| \lesssim \eps$ which implies $|\nabla_{\!y} \tilde{h} - e_n| \lesssim \eps$ after the transformation $(t,x) \mapsto (s,y)$, cf.\ (\ref{perturbed stationary solution bounded from above and below}). Under this transformation applied to cylinders of the form $Q_R(x) = (\frac{R^4}{2},R^4] \times B_R(x)$ we get
\[
Q_R(x) \, \sim \, (\frac{R^4}{2}, R^4] \times B_R(y) \, = \, Q_R(y)
\]
which follows from Lemma \ref{almost isometry}. Now let $(j, l, \alpha) \in \mathcal{CZ}$, for example take $(j,l,|\alpha|)=(0,0,2)$. Then
\[
|Q_R(x)|^{-\frac{1}{p}} R (R+\sqrt{x_n}) \, \| D_{\!x}^2 v \|_{L^p(Q_R(x))} \, \sim \, |Q_R(y)|^{-\frac{1}{p}} R (R+\sqrt{\tilde{h}}) \, \| D_{\!y}^2 \tilde{h} \|_{L^p(Q_R(y))} \, ,
\]
with similar transforms for the other combinations of $j,l$ and $|\alpha|$. The supremum is now taken over all $0<R< \sqrt[4]{T}$ and all $y \in P_s(\tilde{h}) = \{ y \in \R^n \mid \tilde{h}(s,y) > 0 \}$, $s \in I$. Also note that $\tilde{h}$ is controlled by
\[
(1-\tilde{\eps}) \, \tilde{h}(s,y) \; < \; dist\bigl( y,\R^n \setminus spt \, \tilde{h}(s) \bigr) \; < \; (1-\tilde{\eps})^{-1} \, \tilde{h}(s,y)
\]
which follows from a transformation of the statement in Lemma \ref{almost isometry}. \\
All these calculations show that $v^*$ generates a solution $\tilde{h}_1$ via $(t,x) \mapsto (s,y)$ which satisfies the inequality
\[
\sup_{P(\tilde{h})} \, |\nabla_{\!y} \tilde{h}_1 - e_n| + \bigl[ \tilde{h}_1 \bigr]_{X_p^1} \, \lesssim \, \eps \, .
\]

Let $\tilde{h}_2$ be another weak solution. Then, reversing the transformation, we obtain a second solution, say $v^{**}$, of the transformed problem. Thus, by uniqueness of such a solution, $\tilde{h}_1 = \tilde{h}_2$ is a unique solution of
\[
\partial_s \tilde{h}^2 + \nabla_{\!y}\!\cdot\! \bigl(\tilde{h}^2 \; \nabla_{\!y}\Delta_y \tilde{h}^2\bigr) = 0 \, .
\]
Finally, we substitute back for $\tilde{h}=\sqrt{h}$ to see that the initial value problem for the equation (\ref{tfe}) has a unique weak solution, denoted $h^*$. Moreover, since the level set of $h^*$ at height $\lambda \geq 0$ is given by
\[
N_\lambda(h^*) \, = \, graph \, v^*(t,x',\sqrt{\lambda}) \, ,
\]
the analyticity of $N_\lambda(h^*)$ follows directly from Proposition \ref{analyticity in t and x'}. The proof is complete.
\end{proof}


\section{The proofs} 
\label{proofs}
In this section we provide the proofs of section \ref{linear model case} and the proof of Lemma \ref{Y-norm vs X-norm}. We begin with the global energy estimate.

\begin{proof}[Proof of Proposition \ref{energy estimate}]
Similar as for Proposition \ref{energy identity}, we use a regularized version of $\chi_I \partial_t u$ as test function. This gives
\[
\int_I \| \partial_t u \|_{L^2(\mu_1)} \, d\mathcal{L} \, \leq \, \int_I \| f \|_{L^2(\mu_1)} \, d\mathcal{L} \, .
\]
Treating $t$ as a parameter, it is therefore sufficient to consider the elliptic equation $Lu=f$ on $\hs$. To be more precise, in the remaining part of the proof we assume that $u$ satisfies the integral identity
\[
\int_H \Delta u \Delta \varphi \, d\mu_3 + 4 \int_H \nabla'u \cdot \nabla' \varphi \, d\mu_1 \, = \, \int_H f \varphi \, d\mu_1
\]
for all $\varphi \in W^{2,2}(H, \mu_1, \mu_1, \mu_3)$. Now, formally one can prove the energy estimate by testing the elliptic equation with the operator $Lu$ itself. A rigorous justification of this result requires a careful treatment of certain commutators. However, we take a different approach exploiting the fact that the operator $L$ can be factorized as
\[
L \, = \; L_1 L_1 \, ,
\]
where
\[
L_\sigma u \, = \, - x_n^{-\sigma} \, \nabla \!\cdot\! ( x_n^{\,\sigma+1} \, \nabla u) \, , \quad \sigma > -1 \, ,
\]
is the related second order degenerate equation (for a discussion of this equation see \cite{K99}). We begin to study the equation $L_1u=w$ and perform Fourier transformation in the tangential directions $x_1, \dots, x_{n-1}$ to get
\[
x_n \partial_{x_n}^2 \hat{u} + 2 \, \partial_{x_n} \hat{u} - x_n |\xi|^2 \hat{u} \, = \, - \, \widehat{w} \, .
\]
Taking the Fourier variable $\xi \in \R^{n-1}$ as a parameter and putting $z=|\xi| x_n$, this becomes an ODE of the form
\[
\hat{L} \hat{u} \, = \, z \, \partial_z^2 \hat{u} + 2 \, \partial_z \hat{u} - z \hat{u} \, = \, - \, |\xi|^{-1} \widehat{w} \, ,
\]
with $\hat{u}=\hat{u}(\xi,z)$ and $\widehat{w}=\widehat{w}(\xi,z)$. Renaming $\hat{L}=L$, $\hat{u} = u$ and $- |\xi|^{-1}\widehat{w}=w$, we obtain
\begin{equation}
\label{Fourier ODE of second order}
\tag{$\ast$}
L u \, = \, z \, \partial_z^2 u + 2 \, \partial_z u - z u \, = \, w \, ,
\end{equation}
an equation of one independent variable $z \in \R_+$. Now if we substitute $u=z^{-\frac{1}{2}} \, v$ into (\ref{Fourier ODE of second order}), we recover the modified Bessel differential equation $z^2 \partial_z^2 + z \partial_z v - (z^2+\nu^2)v = 0$ of order $\nu=\frac{1}{2}$ for which
\[
I_\nu(z) \, = \, \sum_{j=0}^{\infty} \, \frac{1}{j! \, \Gamma(j+\nu+1)} \, \Bigl( \frac{1}{2} \, z \Bigr)^{2j+\nu}
\]
and
\[
K_\nu(z) \, = \, \frac{1}{2} \, \pi \, \frac{I_{-\nu}(z) - I_\nu(z)}{\sin(\nu\pi)} \, , \quad \nu \notin \Z \, ,
\]
are a fundamental system. Hence a fundamental system for (\ref{Fourier ODE of second order}) is given by $\psi_1(z) = z^{-\frac{1}{2}} \, I_\frac{1}{2}(z)$ and $\psi_2(z) = z^{-\frac{1}{2}} \, K_\frac{1}{2}(z)$. The Wronskian is $\mathcal{W}\bigl( \psi_1(z), \psi_2(z) \bigr) = z^{-2}$ and the operator $T: w \mapsto z^j \, u$ has the kernel
\[
k(z,y) \; = \; \begin{cases}
   - \, y \, \psi_1(z) \, \psi_2(y)   & \text{if} \; z < y \\
   \phantom{-} \, y \, \psi_1(y) \, \psi_2(z)   & \text{if} \; z > y \, .
\end{cases}
\]
From standard ODE theory we know that any solution of (\ref{Fourier ODE of second order}) can be written
\[
z^j \, u(z) \, = \, \int_0^{\infty} z^j \, k(z,y) \, w(y) \, dy
\]
for almost every $z \in \R_+$. Now we check that the kernel conditions
\[
i) \; \sup_{z \in \R_+} \int_0^{\infty} z^j \, |k(z,y)| \, dy \, < \, \infty \qquad \text{and} \qquad 
ii) \; \sup_{y \in \R_+} \int_0^{\infty} z^j \, |k(z,y)| \, \frac{z}{y} \, dz \, < \, \infty
\]
are satisfied provided that $j \in [0,1]$. For this we use $\psi_1(z) \sim 1$ and $\psi_2(z) \sim z^{-1}$ for small $z$, while for large $z$, $\psi_1(z) \sim z^{-1} e^z$ and $\psi_2(z) \sim z^{-1} e^{-z}$. This follows from the corresponding asymptotics of the modified Bessel functions. \\
Then (i) implies that $T: L^{\infty}\bigl( \R_+ \bigr) \to L^{\infty}\bigl( \R_+ \bigr)$, and by (ii) it follows that $T$ maps $L^1\bigl( \R_+, \mu_1 \bigr)$ into itself. Using the Marcinkiewicz interpolation theorem applied to the operator $x_n \, T \, x_n^{-1}$, we thus have
\[
\| u \|_{L^2(\R_+, \mu_{2j+1} )} \, \leq \, c \, \| w \|_{L^2(\R_+, \mu_1)} \, .
\]
In particular, this allows us to control the term $z u$ (the case $j=1$) and therefore, with $v = \partial_z u$, equation (\ref{Fourier ODE of second order}) reduces to
\[
z \, \partial_z v + 2 \, v \, = \, f + z \, u \, =: \, \widetilde{w} \, .
\]
By
\[
\widetilde{k}(z,y) \, = \, y \,
\begin{cases}
   z^{-2}   & \text{if} \; z>y \\
   0  & \text{otherwise}
\end{cases}
\]
the corresponding integral kernel is defined. Now consider the operator $z^\delta \widetilde{w} \mapsto z^\delta v$ for some $\delta > 0$. Since
\[   
\sup_{z \in \R_+} \int_0^{\infty} \Bigl( \frac{z}{y} \Bigr)^\delta |\widetilde{k}(z,y)| \, dy \, = \, \sup_{z \in \R_+} \, z^{\delta-2} \int_0^z y^{1-\delta} \, dy \, = \, \frac{1}{2-\delta} \; , 
\]
it follows
\[
\|  z^\delta \, v\|_{L^{\infty}(\R_+)} \, \lesssim \, \| z^\delta \, \widetilde{w} \|_{L^{\infty}(\R_+)}
\]
for $\delta < 2$. On the $L^1$-side of the estimate we obtain
\[   
\sup_{y \in \R_+} \int_0^{\infty} \Bigl( \frac{z}{y} \Bigr)^\delta |\widetilde{k}(z,y)| \, dz \, = \, \sup_{y \in \R_+} y^{1-\delta} \int_y^{\infty} z^{\delta-2} \, dz \, = \, - \, \frac{1}{\delta-1} \, .
\]
If $\delta < 1$, then
\[
\| z^\delta \, v \|_{L^1(\R_+)} \, \lesssim \, \| z^\delta \, \widetilde{w} \|_{L^1(\R_+)} \, .
\]
Choosing $\delta=\frac{1}{2}$\,, an interpolation between $L^1$ and $L^{\infty}$ gives
\[
\| \partial_z u \|_{L^2(\R_+, \mu_1)} \, = \, \| v \|_{L^2(\R_+, \mu_1)} \, \lesssim \, \| \widetilde{w} \|_{L^2(\R_+, \mu_1)} \, = \, \| w+ z \, u \|_{L^2(\R_+, \mu_1)} \, \lesssim \, \| w \|_{L^2(\R_+, \mu_1)} \, ,
\]
and thus
\[
\| \partial_z^2 u \|_{L^2(\R_+, \mu_3)} \, = \, \| w - 2 \partial_z u + z \, u \|_{L^2(\R_+, \mu_1)} \, \lesssim \, \| w \|_{L^2(\R_+, \mu_1)}
\]
by (\ref{Fourier ODE of second order}). Now recall that $u = \hat{u}$ and $|\xi| w = - \widehat{w}$. A retransformation from $z$ to $x_n$ and integration in $\xi \in \R^{n-1}$ yield
\[
\| |\xi| \, \hat{u} \|_{L^2(H,\mu_1)} + \| |\xi|^2 \, \hat{u} \|_{L^2(H,\mu_3)} + \| \partial_{x_n} \hat{u} \|_{L^2(H,\mu_1)} + \| \partial_{x_n}^2 \hat{u} \|_{L^2(H,\mu_3)} \, \lesssim \, \| \widehat{w} \|_{L^2(H,\mu_1)} \, .
\]
By another application of the estimate, we get
\[
\| \widehat{w} \|_{L^2(H,\mu_1)} \, \lesssim \, \| |\xi|^{-1} \, \hat{f} \|_{L^2(H,\mu_1)}
\]
which is possible since $\hat{L} \widehat{w} = |\xi|^{-1} \widehat{f}$. Now an inverse Fourier transformation converts the inequality into
\[
\| \nabla'\nabla' u \|_{L^2(H,\mu_1)} + \| \nabla' \Delta' u \|_{L^2(H,\mu_3)} + \| \nabla' \partial_{x_n} u \|_{L^2(H,\mu_1)} + \| \nabla' \partial_{x_n}^2 u \|_{L^2(H,\mu_3)} \, \lesssim \, \| f \|_{L^2(H,\mu_1)} \, .
\]
This implies
\[
\| \nabla' \nabla u \|_{L^2(H, \mu_1)} + \| \nabla' D_{\!x}^2 u \|_{L^2(H, \mu_3)} \, \lesssim \, \| f \|_{L^2(H, \mu_1)}
\]
by means of the auxiliary formula (\ref{auxiliary formula}). \\

Higher order derivatives can be obtained in a similar manner: Differentiating (\ref{Fourier ODE of second order}) leads to the modified Bessel equation of order $\nu=1$. Following the same line of argument, we can verify the kernel estimates (i)--(ii) with $k$ replaced by $\partial_z k$, and thus whenever $u$ satisfies the equation $Lu=w$, we get
\[
\| \partial_z u \|_{L^2(\R_+, \mu_3)} \, \lesssim \; \| w \|_{L^2(\R_+, \mu_1)} \, .
\]
This amounts to
\[
\|  \partial_{x_n} \Delta' u \|_{L^2(H, \mu_3)} \, \lesssim \, \| f \|_{L^2(H, \mu_1)}
\]
after renaming back the involved variables and an inverse Fourier transformation. Repeating this procedure we obtain
\[
z \, \partial_z^2(\partial_z^2 u) + 4 \, \partial_z(\partial_z^2 u) - z \, \partial_z^2 u \, = \,  \partial_z^2 w + 2 \, \partial_z u
\]
and hence
\[
\| \nabla \partial_{x_n}^2 u \|_{L^2(H, \mu_3)} + \| D_{\!x}^2 \partial_{x_n}^2 u \|_{L^2(H, \mu_5)} \, \lesssim \, \| f \|_{L^2(H, \mu_1)} \, .
\]
Finally, we claim that
\begin{equation}
\label{fourth tangential derivatives on the Fourier level}
\tag{$\ast\ast$}
\| u \|_{L^2(\R_+, \mu_5)} \, \lesssim \, \| f \|_{L^2(\R_+, \mu_1)} \, .
\end{equation}
Then
\[
\| D_{\!x'}^4 u \|_{L^2(H, \mu_5)} \, \lesssim \, \| f \|_{L^2(H, \mu_1)}
\]
and Proposition \ref{energy estimate} follows by a combination of the above estimates . \\

To prove (\ref{fourth tangential derivatives on the Fourier level}) we consider the equation $z \, \partial_z^2 u + 4 \, \partial_z u - z u = w + 2 \, \partial_z u$ which follows directly from (\ref{Fourier ODE of second order}) by adding $2 \, \partial_z u$ to both sides of the equation. The left hand side transforms into a Bessel equation of order $\nu=\frac{3}{2}$ which implies (\ref{fourth tangential derivatives on the Fourier level}), just as above.
\end{proof}

Next we want to prove local versions of the energy estimate obtained in Proposition \ref{energy estimate}. To this end, we construct cut-off functions in terms of the intrinsic metric $d$. A key insight that arises from (\ref{ball topology}) is the following: If an intrinsic ball is located ``near'' the boundary, then we have $B_R(x) \sim B_{R^2}^{eu}(x)$,\footnote{This relation is to be understood as follows: There exist $0 < c \leq 1 \leq C$ such that $B_{c R^2}^{eu}(x) \subset B_R(x) \subset B_{C R^2}^{eu}(x)$.} while ``far away'' from there we have $B_R(x) \sim B_{R \sqrt{x_n}}^{eu}(x)$. This particular behavior suggests to consider different treatments depending on the ball's position relative to $\partial H$. We start with the derivation of an energy estimate in the latter case, i.e.\ for solutions defined in the parabolic cylinder
\[
Q_R(t,x) \, = \, (t-R^4, t] \times B_R(x)
\]
with $R \ll d(x,\partial H)$.

\begin{lemma}
\label{local energy estimate in 1}
Let $l$ be any nonnegative integer and $\alpha$ any multi-index. If $u$ is a an energy solution of $\partial_t u + Lu = 0$ on $Q_R(0, e_n)$, with $R \ll 1$ and $e_n = (0,\dots,0,1) \in H$, then there exists a small $\delta > 0$ such that
\[
\| \partial_t^l \partial_x^\alpha u \|_{L^2(Q_{\delta R}(0, \, e_n))} \, \leq \, c \, R^{-4l - |\alpha|} \, \| u \|_{L^2(Q_R(0, e_n))}
\]
for some positive constant $c=c(n,l,\alpha)$.
\end{lemma}

\begin{proof}
First we choose a suitable cut-off function $\eta \in C_c^{\infty}\bigl((-R^4,R^4) \times B_R(e_n)\bigr)$ with $\eta \equiv 1$ on $Q_{\delta R}(0, e_n)$ for a sufficiently small $\delta \in (0,1)$. Taking a product ansatz we may additionally assume that
\[
|\partial_t^l \partial_x^\alpha \eta| \, \lesssim \, R^{-4l-|\alpha|} \, .
\]
Now suppose $u$ is an energy solution of $\partial_t u + Lu = f$ on $Q_R(0,e_n)$ with $f \in L^2\bigl( Q_R(0,e_n), \mu_1 \bigr)$, then
\[
(\partial_t + L)(\eta u) \, = \, \eta \, f + (\partial_t \eta + L\eta)u + \omega \qquad \text{on} \quad (-R^4,0) \times \hs
\]
where
\[
\omega \, = \, - 8 \, \nabla' \eta\cdot\nabla' u + 2 \, x_n^{\,2} \, \Delta\eta ,\Delta u + 2 x_n^{-1} \bigl( \nabla(x_n^{\,3} \, \Delta\eta) \cdot \nabla u + \nabla\eta\cdot\nabla(x_n^{\,3} \, \Delta u) + \Delta( x_n^{\,3} \, \nabla\eta \cdot \nabla u \bigr) \, .
\]
To put it concisely, the local solution $u$ becomes a global solution by multiplication by $\eta$. \\
Finally, we observe that $x_n \sim  1$ in $B_R(e_n)$, and hence we can replace any weighted measure by the Lebesgue measure and vice versa. This also reflects the property that locally $\partial_t u + L u = f$ is a uniformly parabolic equation of fourth order. \\

Let $Q_R = Q_R(0,e_n)$. Using the integration by parts formula and the Cauchy-Schwarz inequality, we find
\[
\| \nabla (\eta u) \|_{L^2(Q_R)} + \| D_{\!x}^2 (\eta u) \|_{L^2(Q_R)} \, \lesssim \, R^{-2} \, \| u \|_{L^2(Q_R)} + R^2 \, \| f \|_{L^2(Q_R)}
\]
by means the energy identity in Proposition \ref{energy identity} combined with the auxiliary calculation (\ref{auxiliary formula}). \\
Next, we invoke the Poincar\'e inequality applied to the function $\nabla(\eta u) \in W^{1,2}(B_R)$ to conclude that
\[
\| \nabla(\eta u) \|_{L^2(B_R)} \, \leq \, c(n) R \, \| D_{\!x}^2 (\eta u) \|_{L^2(B_R)} \, .
\]
This is possible because $\eta(t)=0$ on $\partial B_R$ for all $t \in (-R^4,0]$. Altogether this amounts to the estimate
\begin{equation}
\label{1st order spatial energy estimate}
\tag{$\ast$}
R \, \| \nabla u \|_{L^2(Q_{\delta R})} + R^2 \, \| D_{\!x}^2 u \|_{L^2(Q_{\delta R})} \, \lesssim \, \| u \|_{L^2(Q_R)} + R^4 \, \| f \|_{L^2(Q_R)} \, ,
\end{equation}
since also $\eta \equiv 1$ on the smaller cylinder $Q_{\delta R}$. Tangential derivatives commute with the operator $\partial_t + L$. Thus boundedness of all tangential derivatives follows from boundedness of the first order derivatives given in (\ref{1st order spatial energy estimate}) by iteration. \\

To control vertical derivatives as well, we first set $u^{(k)} = \partial_{x_n}^k u$ with $k \in \N_0$. It satisfies the equation
\[
\partial_t u^{(k)} + x_n^{-k-1} \Delta (x_n^{\,k+3} \, \Delta u^{(k)}) \, = \, \tilde{f}_k
\]
in the sense of distributions, where $\tilde{f}_k = \partial_{x_n}^k f - 2k x_n \Delta'\Delta u^{(k-1)} - k(k-1) \Delta' \Delta u^{(k-2)}$. Also note that $\tilde{f}_0=f$. We can do the same calculations for the operator $L_k = x_n^{-k-1} \Delta (x_n^{\,k+3} \, \Delta)$ to obtain analogues of the results established in the Propositions \ref{energy identity} and \ref{energy estimate}. Assuming $f=0$, we can now prove that
\[
R \, \| \nabla u^{(k)} \|_{L^2(Q_{\delta R})} + R^4 \, \| D_{\!x}^2 u^{(k)} \|_{W^{2,2}(Q_{\delta R})} \, \lesssim \, R^{-k} \, \| u \|_{L^2(Q_R)}
\]
by induction over $k = 0, \dots, \alpha_n$. The regularity of $u^{(k)}$ needed in each of the iteration steps follows from the iterated energy identity and the iterated energy estimate in a similar way as (\ref{1st order spatial energy estimate}) follows from the corresponding statements for $k=0$. \\

Finally, solving inductively $\partial_t^j (\partial_t u + L u) = 0$ for $\partial_t^{j+1}u$ and using the bounds for spatial derivatives we get
\[
\| \partial_t^l u \|_{L^2(Q_{\delta R})} \, \lesssim \, \dots \, \lesssim \, R^{-4l} \, \| u \|_{L^2(Q_R)} \, ,
\]
and the lemma follows with a (possibly smaller) scaling factor $\delta < 1$.
\end{proof}

The situation at the boundary is covered by the following lemma.

\begin{lemma}
\label{local energy estimate in 0}
If $u$ is a an energy solution of $\partial_t u + Lu = 0$ on $Q_1(0, 0)$, then there exists a $\delta > 0$ such that
\[
\| \partial_t^l \partial_x^\alpha u \|_{L^2(Q_\delta(0,0), \mu_1)} \, \leq \, c \, \| u \|_{L^2(Q_1(0, 0), \mu_1)}
\]
for a positive constant $c$ depending on $n,l$ and $\alpha$.
\end{lemma}

\begin{proof}
As before we keep the right endpoint of the time interval and the center of the ball stationary, and we merely write $Q_R$ to mean $Q_R(0, 0)$. Both the proof of Lemma \ref{local energy estimate in 1} and the present one show basically the same pattern. The major change concerns the cut-off function: Via a product ansatz we obtain
\[
|\partial_t^l \partial_x^\alpha \eta| \, \lesssim \, 1
\]
for any $l \in \N_0$ and any multi-index $\alpha$, as well as $\eta \equiv 1$ on $Q_\delta$ while $spt \, \eta \subset (-1,1)\times B_1$. Another difference is the behavior of the measure. Near the boundary we have no control of the weight from below, but still from above. More precisely, we know that $x_n < 2$ for $x\in B_1$. \\

By means of these preliminary considerations, we proceed the same way as in the previous proof to find
\[
\| \nabla u^{(k)} \|_{L^2(Q_\delta, \mu_{k+1})} + \| D_{\!x}^2 u^{(k)} \|_{L^2(Q_\delta, \mu_3)} \, \lesssim \, \| u^{(k)} \|_{L^2(Q_1, \mu_{k+1})}
\]
as well as
\[
\| D_{\!x}^2 u^{(k)} \|_{W^{2,2}(Q_\delta, \mu_{k+1}, \mu_{k+3}, \mu_{k+5})} \, \lesssim \, \| u^{(k)} \|_{L^2(Q_1, \mu_{k+1})} \, .
\]
Iteratively, this gives
\[
\| \partial_x^{\alpha'} u \|_{L^2(Q_\delta, \mu_1)} \, \lesssim \, \| u \|_{L^2(Q_1, \mu_1)}
\]
and
\[
\| \nabla u^{(k)} \|_{L^2(Q_\delta, \mu_{k+1})} \, \lesssim \, \| u \|_{L^2(Q_1, \mu_1)} \, .
\]
If $\alpha_n=0$ we are already done. Suppose now $\alpha_n \geq 1$. With Hardy's inequality applied $\alpha_n$ times to $\partial_{x_n}^{\alpha_n} (\psi u)$, where $\psi$ is a spatial cut-off function obeying the above estimate, we obtain
\begin{align*}
\| \partial_{x_n}^{\alpha_n} u \|_{L^2(B_\delta, \, \mu_1)} \,
&\leq \, \| \partial_{x_n}^{\alpha_n} (\psi u) \|_{L^2(H, \mu_1)} \, \lesssim \, \| \partial_{x_n}^{2 \alpha_n} (\psi u) \|_{L^2(H, \mu_{2\alpha_n+1})} \\
&\lesssim \, \sum_{0\leq\beta_n \leq 2\alpha_n} \| (\underbrace{\partial_{x_n}^{2 \alpha_n-\beta_n} \psi}_{|\cdot| \, \lesssim \, 1}) \, (\partial_{x_n}^{\beta_n} u) \|_{L^2(H, \mu_{\beta_n+1})} \, \lesssim \, \| u \|_{L^2(Q_1, \mu_1)}
\end{align*}
for a sufficiently small $\delta < 1$. To bound temporal derivatives we proceed as before. This finishes the proof.
\end{proof}

\begin{rem}
Note that the local energy estimate obtained in Remark \ref{local energy estimate} is independent of the position of the ball $B_R(x_0)$ and it contains Lemma \ref{local energy estimate in 1} and Lemma \ref{local energy estimate in 0} as special cases.
\end{rem}

The pointwise estimate in Proposition \ref{pointwise estimate} will be a consequence of the previous two lemmas combined with the following Morrey-type inequality: Let $\Omega \subseteq H$ open satisfy the cone condition. If $2 k > n$, then
\begin{equation}
\label{Morrey-type inequality}
\| u \|_{L^{\infty}(\Omega)} \, \leq \, c(n,k, \Omega) \, \| u \|_{W^{k,2}(\Omega)}
\end{equation}
for all $u \in W^{k,2}(\Omega)$.

\begin{proof}[Proof of Proposition \ref{pointwise estimate}]
Since the equation $\partial_t u + Lu = 0$ is invariant under translation in any direction except the $x_n$-direction it suffices to consider $t_0=0$ and $x_0 = (0, \dots, 0, x_{0,n})$. Moreover, we recall that $u \circ T_\lambda$ is an energy solution on $T_\lambda^{-1}(Q_R(0,x_0))$, where $T_\lambda$ is defined as in (\ref{linear scaling}), whenever $u$ is a solution on $Q_R(0, x_0)$. \\

First we assume that $R \leq C \sqrt{x_{0,n}}$ for some $C \gg 1$ and take $\lambda = x_{0,n}$ as scaling factor. In this case we can find an $\eps < 1$ such that $r = \frac{\eps R}{C \sqrt{x_{0,n}}} \ll 1$ is a legitimate radius in Lemma \ref{local energy estimate in 1}, and so we get
\begin{align*}
|\partial_{\hat{t}}^l \partial_{\hat{x}}^\alpha u(0, x_0)| \, 
&\lesssim \, x_{0,n}^{-2l-|\alpha|} \sum_{j + |\beta| \leq k} r^{4j+|\beta|-\frac{n+4}{2}} \, \| \partial_t^{l+j} \partial_x^{\alpha+\beta} (u\circ T_{x_{0,n}}) \|_{[L^2(Q_{\delta r}(0, e_n))} \\
&\lesssim \, x_{0,n}^{-2l-|\alpha|} r^{-4l-|\alpha|-\frac{n+4}{2}} \, \| u\circ T_{x_{0,n}} \|_{L^2(Q_r(0, e_n), \mu_1)} \\
&\leq \, x_{0,n}^{-2l-|\alpha|-\frac{n+3}{2}} r^{-4l-|\alpha|-\frac{n+4}{2}} \, \| u \|_{L^2(Q_R(0, x_0), \mu_1)} \, .
\end{align*}
In the first estimate we used inequality (\ref{Morrey-type inequality}) in $Q_{\delta r}(0, e_n)$ applied to the function $\partial_t^l \partial_x^\alpha (u \circ T_{x_{0,n}})$, and in the second one we used Lemma \ref{local energy estimate in 1} and the fact that $x_n \sim 1$ in $B_1(e_n)$. The last line follows from 
\[
| \det \nabla_{t,x} T_{x_{0,n}} |^{-1} \, d\mu_1(x) \, = \, x_{0,n}^{-n-3} \, d\mu_1(\hat{x})
\]
and the enclosure $T_{x_{0,n}}(Q_r(0, e_n)) \subset Q_R(0,x_0)$, see (\ref{ball topology}). Finally, the coefficient can be estimated to be
\[
x_{0,n}^{-2l-|\alpha|-\frac{n+3}{2}} r^{-4l-|\alpha|-\frac{n+4}{2}} \, \lesssim \, R^{-4l-|\alpha|} \sqrt{x_{0,n}}^{-|\alpha|} (R^{n+4} \sqrt{x_{0,n}}^{\,n+2})^{-\frac{1}{2}}
\]
which is bounded above (up to some constant) by $\delta_{l,\alpha}(R, x_0) |Q_R(0, x_0)|_1^{-\frac{1}{2}}$ as stated. \\

In order to investigate the situation at the boundary, we assume $R > C \sqrt{x_{0,n}}$. Now we claim that
\begin{equation}
\label{Morrey in 0}
\tag{$\ast$}
|\partial_{\hat{t}}^l \partial_{\hat{x}}^\alpha u(0, x_0)| \, \lesssim \, \lambda^{-2l-|\alpha|} \sum_{j+|\beta| \leq k} \| \partial_t^{l+j} \partial_x^{\alpha+\beta} (u \circ T_\lambda) \|_{L^2(Q_\delta(0,0))} 
\end{equation}
provided $2k > n+1$. As in the proof of Lemma \ref{local energy estimate in 0}, we apply Hardy's inequality to find this bounded by
\[
\lambda^{-2l-|\alpha|} \sum_{j+|\beta| \leq k} \| \partial_t^{l+j} \partial_x^{\alpha+\beta} (u \circ T_\lambda) \|_{L^2(Q_\delta(0,0), \mu_2)} + \| \nabla_{\!x} \, \partial_t^{l+j} \partial_x^{\alpha+\beta} (u \circ T_\lambda) \|_{L^2(Q_\delta(0,0), \mu_2)} \, .
\]
Since $x_n \lesssim 1$ in $B_\delta(0)$, we can replace $\mu_2$ by $\mu_1$. Thus by Lemma \ref{local energy estimate in 0} and the transformation formula we get
\[
|\partial_{\hat{t}}^l \partial_{\hat{x}}^\alpha u(0, x_0)| \, \lesssim \, \lambda^{-2l-|\alpha|-\frac{n+3}{2}} \, \| u \|_{L^2(T_\lambda(Q_1(0,0)), \mu_1)} \, .
\]
Taking $\lambda = \eps R^2$, we observe that $T_\lambda(Q_1(0,0)) \subset Q_R(0,x_0)$ if $\eps$ is chosen appropriately small. This combined with
\[
R^{-n-3} \, \lesssim \, \bigl( R^{n+4} (R + \sqrt{x_{0,n}})^{n+2} \bigr)^{-\frac{1}{2}} \, \sim \, |Q_R(0, x_0)|_1^{-\frac{1}{2}}
\]
gives the estimate the desired form. \\

To prove (\ref{Morrey in 0}) we invoke inequality (\ref{Morrey-type inequality}) once more, this time applied to $\partial_{\hat{t}}^l \partial_{\hat{x}}^\alpha (u \circ T_\lambda)$ in the cylinder $Q_r(0,0)$, where $r \ll R$ is so chosen that $T_\lambda^{-1}(Q_r(0,x_0)) \subset Q_\delta(0,0)$.
\end{proof}

\begin{proof}[Proof of Proposition \ref{estimate by rough initial data}]
Fix $x \in \hs$ and $t \in I$, with $t \leq 1$, and a constant $C$. Since either $l \neq 0$ or $\alpha \neq 0$, we conclude
\[
\partial_t^l \partial_x^\alpha \bigl( u(t,x) - C \bigr) \, = \, \partial_t^l \partial_x^\alpha u(t,x) \, .
\]
Proposition \ref{pointwise exponential estimate}, together with the fact that $u-C$ is again a solution with $(u-C)(0)=g-C$, then implies
\[
|\partial_t^l \partial_x^\alpha u(t,x)| \, \lesssim \,
\frac{\delta_{l, \alpha}(\sqrt[4]{t},x)}{|B_{\!\sqrt[4]{t}}(x)|_1^{\,\frac{1}{2}}} \; e^{c_n (c_L^{\,2} + c_L^{\,4}) t - \Psi(x)} \, \| e^{\Psi} (g-C) \|_{L^2(\mu_1)} \, .
\]
For $C=g(x)$, we have $|g(y)-C| \leq |x-y| \, \| \nabla g \|_{L^{\infty}(H)} \, = \, |x-y| \, \| g \|_{\dot{C}^{0,1}}$. Now for a fixed radius $R \in (0,1]$, we decompose $H$ into the annuli $A_i(x) = B_{iR}(x) \setminus B_{(i-1)R}(x)$, $i\in\N$. This gives the estimate
\[
\| e^{\Psi} (g-C) \|_{L^2(\mu_1)} \, \leq \, \Bigl( \ \sum_{i\in\N} \int_{A_i(x)} e^{2 \Psi(y)} \, |x-y|^2 \, d\mu_1(y) \Bigr)^\frac{1}{2} \, \| g \|_{\dot{C}^{\, 0,1}(H)} \, .
\]
We choose $\Psi$ so that $\Psi \sim - \frac{1}{R} \, d(x,\cdot)$. Consequently, $\Psi(x)=0$ and $c_L = \frac{1}{R} \geq 1$. Then we refer to the properties of the intrinsic metric (\ref{ball topology})--(\ref{ball measure}) and the doubling condition (\ref{doubling condition}) to find
\begin{align*}
\bigl( \, \sum_{i\in\N} \int_{A_i(x)} e^{2 \Psi(y)} \, |x-y|^2 \, d\mu_1(y) \bigr)^\frac{1}{2} \,
&\lesssim \, R (R + \sqrt{x_n}) \sum_{i \in \N} e^{-i+1} \, i^{2} \, |B_{iR}(x)|_1^{\,\frac{1}{2}} \\
&\lesssim \, R \, (R + \sqrt{x_n}) \, |B_{R}(x)|_1^{\,\frac{1}{2}} \sum_{i \in \N} e^{-i} \, i^{n+3} \, .
\end{align*}
Since $e^{-i}$ goes to zero as $i \to \infty$ faster than any polynomial, the sum is bounded above. Now, as usual, we set $R=\sqrt[4]{t}$ such that the exponential factor $e^{c_n c_L^{\,4} t - \Psi(x)} = e^{c_n}$ is a constant. Altogether, we obtain
\[
|\partial_t^l \partial_x^\alpha u(t,x)| \, \lesssim \, \sqrt[4]{t}^{1-4l-|\alpha|} (\sqrt[4]{t} + \sqrt{x_n})^{1- |\alpha|} \, \| g \|_{\dot{C}^{0,1}}
\]
for all $(t,x) \in (0, \min\{1,T\}] \times\hs$. \\

The general case follows from this estimate by scaling. Let  $(\hat{t}, \hat{x}) \in I \times \hs$ be arbitrary, but fixed. Then
\begin{align*}
|\partial_{\hat{t}}^l \partial_{\hat{x}}^\alpha u(\lambda^2, \hat{x})| \,
&= \, \lambda^{-2l-|\alpha|} \, |\partial_t^l \partial_x^\alpha (u \circ T_\lambda)(1,x)| \\
&\lesssim \, \lambda^{-2l-|\alpha|} (1 + \sqrt{x_n})^{1-|\alpha|} \, \| \nabla_{\!x} \, g(\lambda\,\cdot) \|_{L^\infty(H)} \\
&= \, \sqrt{\lambda}^{1-4l-|\alpha|} (\sqrt{\lambda} + \sqrt{\hat{x}_n})^{1-|\alpha|} \, \| \nabla_{\!\hat{x}} \, g \|_{L^\infty(H)}
\end{align*}
with $0 < \lambda \leq \sqrt{T}$, and the statement follows with $\lambda = \sqrt{\hat{t}}$.
\end{proof}

We are now ready to prove the crucial Gaussian estimate for the Green kernel.

\begin{proof}[Proof of Theorem \ref{Gaussian estimate}]
It suffices to consider the case $0 = s < t$ which we will assume in the following whenever it is convenient. Now let $u$ be an energy solution of the initial value problem with $u(0)=g$. Then, in view of Corollary \ref{pointwise estimate by initial data}, the functional that assigns to $u(0, \cdot)$ the evaluation $\partial_t^l \partial_x^\alpha u(t,x)$ is continuous. By the Riesz representation theorem, we can thus find a kernel $k_{l, \alpha}(t, x, 0; \cdot) \in L^2(\mu_1)$ such that
\[
\partial_t^l \partial_x^\alpha u(t,x) \, = \, \int_H k_{l, \alpha}(t, x, 0; y) u(0,y) \, d\mu_1(y) \, .
\]
Putting $G_{l, \alpha}(t, x, 0, y) = y_n \, k_{l, \alpha}(t, x, 0; y)$, this already proves the existence of $G=G_{0,0}$. Moreover, since
\[
\partial_t^l \partial_x^\alpha u(t,x) \, = \, \partial_t^l \partial_x^\alpha \int_H G(t,x,0,y) u(s,y) \, dy \, = \, \int_H \partial_t^l \partial_x^\alpha \, G(t,x,0,y) u(0,y) \, dy
\]
by Lebegue's dominated convergence theorem, we have $G_{l, \alpha} = \partial_t^l \partial_x^\alpha G$. \\

Now fix $t > 0$ and $x \in \hs$, and let $\Psi$ be Lipschitz satisfying the assumptions of Lemma \ref{exponential norm decrease}. We have
\begin{equation}
\label{duality for Linfty}
\tag{$\ast$}
\| e^{-\Psi} |B_{\!\sqrt[4]{t}}(\cdot)|_1^{\,\frac{1}{2}} k_{l, \alpha} \|_{L^{\infty}} \, = \, \sup_{\| g \|_{L^1( \mu_1)} \, \leq \, 1} |\int_H k_{l, \alpha} \, e^{-\Psi} |B_{\!\sqrt[4]{t}}(\cdot)|_1^{\,\frac{1}{2}} g \, d\mu_1|
\end{equation}
by definition. By the property that a solution can be written in terms of its corresponding kernel, this reads
\[
\sup_{\| g \|_{L^1( \mu_1)} \, \leq \, 1} \bigl\{ |\partial_t^l \partial_x^\alpha v(t,x)| \mid v \; \text{is solution with} \; v(s) = e^{-\Psi} |B_{\!\sqrt[4]{t-s}}(\cdot)|_1^{\,\frac{1}{2}} g \bigr\} \, .
\]
Hence
\[
e^{\Psi(x)} \, |B_{\!\sqrt[4]{t}}(x)|_1^{\,\frac{1}{2}} \, |\partial_t^l \partial_x^\alpha v(t,x)| \, \lesssim \, \delta_{l, \alpha}(\sqrt[4]{t}, x) \, e^{c_n (c_L^{\,2} +c_L^{\,4}) t} \, \| e^{\Psi} \, v\Bigl( \frac{t}{2} \Bigr) \|_{L^2(\mu_1)}
\]
which follows from Proposition \ref{pointwise exponential estimate} applied to $v$ in the points $s=\frac{t}{2} < t$. \\

To prepare the next step, we introduce the multiplication operator $M: L^2(\mu_1) \to L^2(\mu_1)$ that assigns to a function its multiplication by $|B_{\!\sqrt[4]{t}}(\cdot)|_1^{\,\frac{1}{2}}$ and, given an energy solution $v$, the modified solution operator $\widetilde{S}_s(t): L^2(\mu_1) \ni e^{-\Psi} v(s) \mapsto e^{-\Psi} v(t) \in L^2(\mu_1)$. In these notations, we apply Proposition \ref{pointwise exponential estimate} once more, but now in the points $s=0 < \frac{t}{2}$, for $l=\alpha=0$ and with $\Psi$ replaced by $-\Psi$, to find
\[
\| M \widetilde{S}_0\Bigl( \frac{t}{2} \Bigr) \, e^{-\Psi} v(0) \|_{L^{\infty}} \, \lesssim \, e^{c_n (c_L^{\,2} + c_L^{\,4}) t} \, \| e^{-\Psi} v(0) \|_{L^2(\mu_1)} \, .
\]
Thus $M \widetilde{S}_0$ is also an operator from $L^2(\mu_1)$ to $L^{\infty}$ with operator norm bounded by $c(n) \, e^{c_n (c_L^{\,2} + c_L^{\,4}) t}$. Its dual operator is
\[
(M \widetilde{S}_0)^*:  (L^{\infty})' \, \supset \, L^1(\mu_1) \ni \tilde{g} \mapsto {\widetilde{S}_0}^* M \, \tilde{g} \in L^2(\mu_1)
\]
and the operator norms coincide. Now choosing $\tilde{g} = e^\Psi |B_{\!\sqrt[4]{t}}(\cdot)|_1^{-\,\frac{1}{2}} v(0) \in L^1(\mu_1)$, all this amounts to
\[
\| e^{\Psi} v\Bigl( \frac{t}{2} \Bigr) \|_{L^2(\mu_1)} \, = \, \| \widetilde{S}_0^* \Bigl( \frac{t}{2} \Bigr) M \tilde{g} \|_{L^2(\mu_1)} \, \lesssim \, e^{c_n (c_L^{\,2} + c_L^{\,4}) t} \; \| \tilde{g} \|_{L^1(\mu_1)} \, .
\]
Consequently
\[
e^{\Psi(x)} \, |B_{\!\sqrt[4]{t}}(x)|_1^{\,\frac{1}{2}} \, |\partial_t^l \partial_x^\alpha v(t,x)| \, \lesssim \, \delta_{l, \alpha}(\sqrt[4]{t},x) \, e^{c_n (c_L^{\,2} + c_L^{\,4}) t} \, ,
\]
where we also used $v(0)=e^{-\Psi} |B_{\!\sqrt[4]{t}}(\cdot)|_1^{\,\frac{1}{2}} g$. Finally, we invoke the duality identity (\ref{duality for Linfty}) to reach the estimate
\[
|\partial_t^l \partial_x^\alpha G(t, x, 0, y)| \, \lesssim \, \frac{\delta_{l, \alpha}(\sqrt[4]{t},x) \, y_n}{|B_{\!\sqrt[4]{t}}(x)|_1^{\,\frac{1}{2}} \, |B_{\!\sqrt[4]{t}}(y)|_1^{\,\frac{1}{2}}} \; e^{c_n (c_L^{\,2} + c_L^{\,4}) t + \Psi(y) - \Psi(x)}
\]
valid for almost every $y \in \hs$. \\

Here we specify the choice of the Lipschitz function and set $\Psi(x) = c_L d(x,y)$. Assuming $c_L \geq 1$ and fixing all the other variables, the Gaussian function attains its minimum if $c_L= \bigl( \frac{d(x,y)}{4 c_n t} \bigr)^\frac{1}{3}$. Indeed,
\[
- \bigl( c_L d(x,y) - c_n c_L^{\,4} \, t \bigr) \, \geq \, - \, \frac{d(x,y)^\frac{4}{3}}{(c_n t)^\frac{1}{3}} \, \underbrace{(4^{-\frac{1}{3}} - 4^{-\frac{4}{3}})}_{> 0} \, = \, - \, c_{n}^{-1} \, \Bigl(\frac{d(x,y)^4}{t}\Bigr)^\frac{1}{3}
\]
and the theorem follows. \\

To rule out the quadratic term in the exponential factor, we have to assume that $c_L$, which has been calculated to be
\[
c_L \, = \, \Bigl( \frac{d(x,y)}{4 c_n t} \Bigr)^\frac{1}{3} \, ,
\]
is greater or equal than $1$. In the opposite situation when $c_L < 1$, we thus have $d(x,y) < 4 c_n t$. Scaling further reduces the time length scale to $t=1$ (cf.\ the proof of Proposition \ref{estimate by rough initial data}). Now following the same line of argument as above, but with Corollary \ref{pointwise estimate by initial data} instead of Proposition \ref{pointwise exponential estimate}, we reach
\[
|\partial_t^l \partial_x^\alpha G(t, x, 0, y)| \, \lesssim \, \frac{\delta_{l, \alpha}(\sqrt[4]{t},x) \, y_n}{|B_{\!\sqrt[4]{t}}(x)|_1^{\,\frac{1}{2}} \, |B_{\!\sqrt[4]{t}}(y)|_1^{\,\frac{1}{2}}} \; .
\]
However, this implies (\ref{ge}) since
\[
1 \, < \, e^{1 - \bigl( \frac{d(x,y)}{4 c_n} \bigr)^\frac{4}{3}}
\]
in this case and the proof is complete.
\end{proof}

We can use the exponential decay in $d$ to prove the kernel estimates in Proposition \ref{CZ-kernel estimate} which are necessary to apply the Calder\'on-Zygmund theory on $(I \times \hs, d_0, \mathcal{L} \times \mu_1)$.

\begin{proof}[Proof of Proposition \ref{CZ-kernel estimate}]
Let $s < t$, and $j,l, \alpha$ be as in the proposition. The Gaussian estimate (\ref{ge}) implies
\begin{equation}
\label{kernel estimate by volume}
\tag{$\ast$}
y_n^{-1} x_n^{\,j} \, |\partial_t^l \partial_x^\alpha G(t,x,s,y)| \, \lesssim \, d_0^{\,2j-4l-2|\alpha|} \, \bigl( |B_{d_0}(x)|_1 + |B_{d_0}(y)|_1 \bigr)^{-1} \, .
\end{equation}
This can be easily checked in two consecutive steps: Putting $R = \sqrt[4]{t-s} > 0 $, we first observe that
\[
x_n^{\,j} \, \delta_{l,\alpha}(R,x) \, \leq \, R^{\,2j-4l-2|\alpha|}
\]
if $2j - |\alpha| \leq 0$. Now, since $R^{-m} \bigl( R + d(x,y) \bigr)^m e^{- \eps \bigl( \frac{d(x,y)}{R} \bigr)^\frac{4}{3}} \leq c(m)$ for all $m\geq 0$ and all $\eps > 0$, this gives
\[
R^{2j-4l-2|\alpha|} \, e^{-\eps \bigl( \frac{d(x,y)}{R} \bigr)^\frac{4}{3}} \, \lesssim \, \bigl( R + d(x,y) \bigr)^{2j-4l-2|\alpha|} \, \leq \, d_0^{\,2j-4l-2|\alpha|} \, .
\]
Second, 
\[
|B_R(x)|_1^{-1} \, \lesssim \, \Bigl( 1 + \frac{d(x,y)}{R} \Bigr)^{2n+2} \, \bigl( |B_R(x)|_1 + |B_R(y)|_1 \bigr)^{-1}
\]
\[
\lesssim \, \Bigl( 1 + \frac{d(x,y)}{R} \Bigr)^{4n+4} \, \bigl( |B_{d_0}(x)|_1 + |B_{d_0}(y)|_1 \bigr)^{-1} \, ,
\]
just as in Remark \ref{exchange of centers in ge}. Again, the prefactor is absorbed by the exponential function which in turn is bounded by $1$. This proves (\ref{kernel estimate by volume}) and consequently, if $\frac{m}{2} = 2l+|\alpha|-j = 2$, we reach the first kernel estimate
\[
|K(t,x,s,y)| \, \lesssim \, d_0^{-4} \, \bigl( |B_{d_0}(x)|_1 + |B_{d_0}(y)|_1 \bigr)^{-1} \, \lesssim \, V(t,x,s,y)^{-1} \, .
\]
Also note that both conditions, $\frac{m}{2} = 2$ and $|\alpha| \geq 2j$, are satisfied if and only if $(j,l,\alpha) \in \mathcal{CZ}$.\\

For the second kernel condition we first observe that
\[
|K(t,x,s,y) - K(\bar{t},\bar{x},\bar{s},\bar{y})| \, \leq \, |K(t,x,s,y) - K(\bar{t},x,s,y)| + |K(\bar{t},x,s,y) - K(\bar{t},\bar{x},s,y)| \, + 
\]
\[
+ |K(\bar{t},\bar{x},s,y) - K(\bar{t},\bar{x},\bar{s},y)| + |K(\bar{t},\bar{x},\bar{s},y) - K(\bar{t},\bar{x},\bar{s},\bar{y})| \, = \, (I) + (II) + (III) + (IV) \, .
\]
We estimate term by term. We use the fundamental theorem of calculus and argue as for (\ref{kernel estimate by volume}) to find
\[
(I) \, \leq \, |t-\bar{t}| \, \sup_\tau |\partial_\tau K(\tau,x,s,y)| \, \lesssim \, |t-\bar{t}| \, \sup_\tau \, \bigl( |\tau-s| + d(x,y)^4 \bigr)^{-1} \, V(\tau,x,s,y)^{-1} \, ,
\]
if $(j,l,\alpha) \in \mathcal{CZ}$. The supremum is taken over all $\tau$ between $t$ and $\bar{t}$. By the assumption $D \leq \frac{1}{6}$, we have 
\[
|t-\bar{t}| \, \lesssim \, d_0^{\,3} \quad \text{and} \quad d_0 \, \sim \, d^{(t)}\bigl( (\bar{t},\bar{x}),(\bar{s},\bar{y}) \bigr) \, ,
\]
and hence $(I) \lesssim D V^{-1}$. In a similar way we get $(III) \lesssim D V^{-1}$, where we also need Lemma \ref{Gaussian estimate involving s- and y-derivatives} to control one derivative in the $s$-variable.\\

Now suppose $\gamma: [a,b] \to \hs$ is the geodesic between $x$ and $\bar{x}$, that is, the length of $\gamma$ is $d(x,\bar{x})$. Then
\[
\bigl| \int_a^b \nabla_{\!\gamma(\tau)\,} K\bigl( \bar{t},\gamma(\tau),s,y \bigr) \, \gamma'(\tau) \, d\tau \bigr| \, \leq \, d(x,\bar{x}) \, \sup_{z\in(\gamma)} \sqrt{z_n} \, |\nabla_{\!z} K(\bar{t},z,s,y)| \, ,
\]
again by the fundamental theorem of calculus. Employing the assumption $D \leq \frac{1}{6}$ once more, we see that
\[
d(x,y) \, \lesssim \, \sqrt[4]{|\bar{t}-s|+d(z,y)^4} \qquad \forall \; z \in \bar{B}_{d(x,\bar{x})}(x)
\]
such that
\[
e^{- \eps \bigl( \frac{d(z,y)^4}{|\bar{t}-s|} \bigr)^\frac{1}{3}} \, \lesssim \, e^{- \tilde{\eps} \bigl( \frac{d(x,y)^4}{|\bar{t}-s|} \bigr)^\frac{1}{3}} \qquad (0<\tilde{\eps} \leq \eps) \, .
\]
Moreover, we have $d_0 \lesssim \sqrt[4]{|\bar{t}-s|+d(x,y)^4}$ provided that $D \leq \frac{1}{6}$. Clearly, $(\gamma) \subset \bar{B}_{d(x,\bar{x})}(x)$ such that
\[
(II) \, \lesssim \, d(x,\bar{x}) \, V(\bar{t},x,s,y)^{-1} \, \sup_{z\in(\gamma)} \sqrt{z_n} \, \delta_{0,1}(\sqrt{|\bar{t}-s|},z) \, e^{- \frac{\tilde{\eps}}{2} \bigl( \frac{d(x,y)^4}{|\bar{t}-s|} \bigr)^\frac{1}{3}}
\]
by virtue of (\ref{kernel estimate by volume}). Now using the exponential decay, this implies $(II) \lesssim D V^{-1}$. For the final estimate we argue similarly, but now with Lemma \ref{Gaussian estimate involving s- and y-derivatives} instead of Theorem \ref{Gaussian estimate}. The proposition is proved.
\end{proof}

The proof of the linear estimate in Proposition \ref{X-norm vs Y-norm} consist of two parts, an on-diagonal estimate (Proposition \ref{weighted Lp-estimate}) and an off-diagonal estimate. The following lemma deals with the latter.

\begin{lemma}
\label{pointwise estimate against f with support f outside of a cylinder}
Let $I=(0,1)$, $x_0 \in \hs$, $p \geq 1$ and $f \in Y_p$ with $spt \, f \subseteq \bigl( [0,1] \times \hs \bigr) \setminus \bigl( (\frac{1}{4},1] \times B_2(x_0) \bigr)$. We further suppose that $j \geq 0$, $l \in \N_0$, $\alpha$ is any multi-index satisfying $|\alpha| \geq 2j$. If $u$ is an energy solution of the inhomogeneous equation on $[0,1) \times \hs$ with initial condition $u(0)=0$, then the estimate
\[
x_n^{\,j} |\partial_t^l \partial_x^\alpha u(t,x)| \, \leq \, c(n,j,l,\alpha) \, (1+\sqrt{x_{0,n}})^{2j+1-|\alpha|} \, \| f \|_{Y_p}
\]
holds for all $(t,x) \in (\frac{1}{2},1] \times B_1(x_0) \subseteq (0,1] \times \hs$.
\end{lemma}

\begin{proof}
Let $Q(x_0) = (\frac{1}{4},t] \times B_2(x_0)$ for which $Q(x_0) \cap spt \, f = \emptyset$. Since $B_1(x) \subset B_2(x_0)$ for $x \in B_1(x_0)$,
\[
\bigl( (0,t] \times H \bigr) \setminus Q(x_0) \, \subset \, \bigl( (0,t] \times H \bigr) \setminus \bigl( (\frac{1}{4},t] \times B_1(x) \bigr)
\]
such that
\[
x_n^{\,j} |\partial_t^l \partial_x^\alpha u(t,x)| \, \lesssim \, \int_{(0,1) \times H} (1+\sqrt{y_n})^{2j-|\alpha|} \, |B_1(y)|^{-1} \, e^{- \frac{d(x,y)}{4 c_n}} \, |f(s,y)| \, dy ds
\]
for all $(t,x) \in (\frac{1}{2}, 1] \times B_1(x_0)$ by Lemma \ref{Gaussian estimate outside of a cylinder}. Thanks to the exponential decay in $d(x,y)$ we may replace $y_n$ by $x_n$ in the first factor, regardless of the sign of the exponent. Thus we obtain the bound
\[
(1+\sqrt{x_{0,n}})^{2j+1-|\alpha|} \int_{(0,1) \times H} (1+\sqrt{y_n})^{-1} \, |B_1(y)|^{-1} \, e^{- \frac{d(x,y)}{8 \, c_n}} \, |f(s,y)| \, dy ds
\]
since also
\[
(1+\sqrt{x_n})^{\gamma} \, \lesssim \, (1+d(x,x_0))^{2 |\gamma|} \, (1+\sqrt{x_{0,n}})^{\gamma} \, < \, 2^{2 |\gamma|} \, (1+\sqrt{x_{0,n}})^{\gamma}
\]
for all $x \in B_1(x_0)$, and hence it remains to estimate the integral. To this end, we first cover the upper half space by countably many balls $B_1(y_0)$. Then the above integral is (up to a constant) bounded by
\[
\sum_{y_0} \int_0^1 \int_{B_1(y_0)} (1+\sqrt{y_n})^{-1} \, |B_1(y)|^{-1} \, |f(s,y)| \, e^{-\frac{d(x,y)}{8 c_n}} \, dy ds
\]
\[
\lesssim \, \Bigl( \sup_{y_0} \int_0^1 \int_{B_1(y_0)} (1+\sqrt{y_n})^{-1} \, |B_1(y)|^{-1} \, |f(s,y)| \, dy ds \, \Bigr) \, \sum_{y_0} \, e^{- \frac{d(x,y_0)}{8 c_n}} \, .
\]
The series is uniformly convergent in $x$ because
\[
\sum_{y_0} e^{- \frac{d(x,y_0)}{8 c_n}} \, \leq \, \sum_{k \in \N} \, \sum_{y_0 \in B_k(x)} e^{- \frac{k-1}{8 c_n}} \, = \, e^\frac{1}{8 c_n} \, \sum_{k \in \N} e^{-\frac{k}{8 c_n}} \, \# \{y_0 \mid y_0 \in B_k(x)\} \, ,
\]
and the number of lattice points\footnote{This is related to the \emph{Gauss circle problem} that asks how many lattice points are inside a given ball of radius $k$. In the $2$-dimensional Euclidean setting there are about $N(k) = \pi k^2 + \mathcal{O}(k^{0.5+\eps})$, with $0<\eps\leq 0.1298\dots$, integer lattice points in $B_{k}(0)$. The lower limit $0$ was obtained independently by Hardy and Landau in 1915, and the upper bound by Huxley \cite{H03}.} in a ball grows at most polynomially in $k$. As for the time interval $(0,1]$, we choose the cover $(\frac{1}{2} \, R_m^{\,4}, R_m^{\,4}]$, $m \in N_0$, where $R_m = 2^{-\frac{m}{4}}$. Moreover, for any $y_0\in \hs$ and any $m \in \N_0$, there exists a disjoint collection of balls $\{ B_{\!\frac{1}{3} R_m}(z_i) \}_{i=1}^N$ in $B_1(y_0)$ such that
\[
\bigcup_{i=1}^N B_{R_m}(z_i) \, \supset \, B_1(y_0) \, .
\]
This covering result is due to Vitali (see e.g.\ \cite{K04}). Hence, invoking also the doubling property (\ref{doubling condition}), we find
\begin{equation}
\label{Vitali's consequence}
\tag{$\ast$}
\sum_{i=1}^N \, |B_{R_m}(z_i)| \, \leq \, b \sum_{i=1}^N |B_{\!\frac{1}{3} R_m}(z_i)| \, = \, b \, |\bigcup_{i=1}^N B_{\!\frac{1}{3} R_m}(z_i)| \, \leq \, b \, |B_1(y_0)| \, .
\end{equation}
Note that the cylinders $Q_{R_m}(z_i)$, $i=1, \dots, N$, cover $(1,0] \times B_1(y_0)$ and we thus investigate the expression
\[
(I) \, = \, \sup_{y_0} \sum_{m\in\N_0} \, \sum_{i=1}^N \int_{Q_{R_m}(z_i)} (1+\sqrt{y_n})^{-1} \, |B_1(y)|^{-1} \, |f(s,y)| \, dy ds \, .
\]
The idea now is to consider the boundary-case and the situation away from the boundary separately. Write 
\[
\{ y_0 \} \, = \, \{\sqrt{y_{0,n}} \ll 1\} \cup \{1 \ll \sqrt{y_{0,n}}\} \, = \, A \, \cup \, A' \, .
\]
We start by discussing the latter case $y_0 \in A'$. If $y_{0,n}$ is sufficiently large in comparison with $1$, we already know that $y_{0,n} \sim y_n$ for all $y \in B_1(y_0)$ and hence, in particular, for $y=z_i$. Thus by transitivity,
\[
(I) \, \lesssim \,
\sum_{m\in\N_0} \, \sum_{i=1}^N |B_1(y_0)|^{-1} \, (R_m + \sqrt{z_{i,n}})^{-1} \int_{Q_{R_m}(z_i)} |f(s,y)| \, dy ds \, .
\]
Finally, using H\"older's inequality and the estimate obtained in (\ref{Vitali's consequence}), we see that this is bounded above by
\[
|B_1(y_0)|^{-1} \sum_{m\in\N_0} \, \sum_{i=1}^N (R_m + \sqrt{z_{i,n}})^{-1} \, |Q_{R_m}(z_i)|^{\frac{p-1}{p}} \, \| f \|_{L^p(Q_{R_m}(z_i))}
\]
\[
\leq \, |B_1(y_0)|^{-1} \sum_{m\in\N_0} R_m \sum_{i=1}^N |B_{R_m}(z_i)| \, \| f \|_{Y_p} \, \lesssim \, \sum_{m\in\N_0} R_m \, \| f \|_{Y_p}
\]
for all $p \geq 1$. \\

Turning to the case $y_0 \in A$, we apply the rather rough estimate $(1+\sqrt{y_n})^{-1} \, |B_1(y)|^{-1} \lesssim 1$ to get
\[
(I) \, \lesssim \, \sum_{m\in\N_0} \, \sum_{i=1}^{N} R_m \, (R_m + \sqrt{z_{i,n}}) \, |B_{R_m}(z_i)| \, \| f \|_{Y_p} \,\overset{\text{(\ref{Vitali's consequence})}}{\lesssim} \, |B_1(y_0)| \, \sum_{m\in\N_0} R_m \, \| f \|_{Y_p} \, .
\]
Now, near the boundary we can always find a ball centered on the boundary that contains $B_1(y_0)$ such that $|B_1(y_0)|$ is bounded by some finite number which does not depend on the location of $y_0$. Summation over $m \in \N_0$ is possible in any of the cases and so, for any $(t,x) \in (\frac{1}{2},1] \times B_1(x_0)$, we reach
\[
x_n^{\,j} |\partial_t^l \partial_x^\alpha u(t,x)| \, \lesssim \, (1+\sqrt{x_{0,n}})^{2j+1-|\alpha|} \, \| f \|_{Y_p} \sum_{m\in\N_0} R_m \, .
\]
This amounts to the assertion of the lemma.
\end{proof}

This implies the linear estimate $\| u \|_{X_p} \lesssim \| f \|_{Y_p} + \| g \|_{\dot{C}^{0,1}}$ which is an essential ingredient in the proof of our main result for the perturbation equation obtained in Theorem \ref{well-posedness of the nonlinear problem}.

\begin{proof}[Proof of Proposition \ref{X-norm vs Y-norm}]
We first consider the case that $spt \, f \subseteq [0,1] \times \hs$. For an arbitrary $x_0 \in \hs$, we write 
\[
f \, = \, \chi_{Q(x_0)} f + (1-\chi_{Q(x_0)}) f \, = \, f_1 + f_2 \, ,
\]
where $Q(x_0) = (\frac{1}{4},1] \times B_2(x_0)$. This splits the equation into an on- and an off-diagonal part. \\

Now suppose $u$ is an energy solution of $\partial_t u + Lu=f_1$ with $u(0)=0$. For $p \in (1,\infty)$, we then have
\[
\| \partial_t^l \partial_x^\alpha u \|_{L^p(Q_1(x_0), \mu_{jp})} \, \lesssim \, \| f_1 \|_{L^p((0,1) \times H)} \, \lesssim \, (1+\sqrt{x_{0,n}}) \, | Q_1(x_0)|^\frac{1}{p} \, \| f \|_{Y_p}
\]
by Proposition \ref{weighted Lp-estimate} with $\sigma=0$. For the first estimate we require $j,l$ and $\alpha$ to be admissible, that is such that $(j,l,\alpha) \in \mathcal{CZ}$. For the second one we cover $Q(x_0)$ by cylinders of the form $Q_{r_k}(x_i)$ to get
\[
\| f \|_{L^p(Q(x_0))} \, \leq \, \sum_{i=1}^{c(n)} \, \sum_{k=1}^2 r_k^{-3} \, (r_k+\sqrt{x_{i,n}}) \, |Q_{r_k}(x_i)|^\frac{1}{p} \, \| f \|_{Y_p} \, ,
\]
with $r_k \in \{2^{-\frac{1}{4}},1\}$ and $x_i \in B_2(x_0)$ for $1\leq i \leq c(n)$. Again there are different treatments near and far away from the boundary: Let $\sqrt{x_{0,n}} \lesssim 1$ and $|\alpha| \geq 2j$. Then, through the obvious inequality
\[
1 + \sqrt{x_{0,n}} \, \lesssim \, (1+\sqrt{x_{0,n}})^{2j+1-|\alpha|} \, ,
\]
the above estimate takes on the desired form. \\
For $1\lesssim \sqrt{x_{0,n}}$, it requires a different ansatz to close the gap between the factor at hand, $1 + \sqrt{x_{0,n}}$, and the one to the power $2j+1-|\alpha|$. As mentioned previously, $x_n \sim x_{0,n}$ for all $x\in B_2(x_0)$, and thus
\[
(1+\sqrt{x_{0,n}})^{|\alpha|-2j} \, \| \partial_x^\alpha u \|_{L^p(Q_1(x_0),\mu_{jp})} \, \lesssim \, \| x_n^{\,\frac{|\alpha|}{2}} \, \partial_x^\alpha u \|_{L^p((0,1) \times H)} \, \lesssim \, \| f \|_{L^p(Q(x_0))}
\]
if $|\alpha|<4$ by virtue of Lemma \ref{Lp boundedness for Schur operators}. (For $|\alpha|=0,1,4$, there is nothing to prove.) It remains to show that $\nabla u$ has a pointwise bound. Using that $\nabla u$ can be written in terms of the Green function we get
\[
|\nabla u(1,x_0)| \, \leq \, \| \nabla_{\!x} G(1,x_0,\cdot,\cdot) \|_{L^\frac{p}{p-1}((0,1) \times H)} \, \| f \|_{L^p(Q(x_0))}
\]
by H\"older's inequality. To bound the norm containing the Green function, we apply Lemma \ref{certain derivatives of Phi are in Lq} to find
\[
\| \nabla_{\!x} G(1,x_0,\cdot,\cdot) \|_{L^\frac{p}{p-1}((0,1) \times H)} \, \lesssim \, 2^{-\frac{1}{p}} \, (1+\sqrt{x_{0,n}})^{-1} \, |Q_1(x_0)|^{-\frac{1}{p}} \, .
\]
This is possible if $\frac{p}{p-1} < \frac{n+2}{n+1}\,$, that is for $p > n+2$, and the desired bound follows. \\

Now let $u$ be a solution of $\partial_t u + Lu=f_2$ with vanishing initial value. We apply Lemma \ref{pointwise estimate against f with support f outside of a cylinder} to obtain
\[
\| \partial_t^l \partial_x^\alpha u \|_{L^p(Q_{1}(x_0),\mu_{jq})} \, \lesssim \, (1+\sqrt{x_{0,n}})^{2j+1-|\alpha|} \, |Q_1(x_0)|^\frac{1}{q} \, \| f \|_{Y_p} \, ,
\]
if $|\alpha| \geq 2j$. Given $(j,l,\alpha) \in \mathcal{CZ}$, this last condition is always true. In addition, with $|\alpha| = 1$ and $j=l=0$, Lemma \ref{pointwise estimate against f with support f outside of a cylinder} yields
\[
|\nabla u(t,x)| \, \lesssim \, \| f \|_{Y_p}
\]
for all $(t,x) \in Q_{1}(x_0) = (\frac{1}{2},1] \times B_{1}(x_0)$ and hence, in particular, for $(t,x)=(1,x_0)$. \\

Altogether we have seen that for all $p>n+2$ and $(j,l, \alpha) \in \mathcal{CZ}$ the following two estimates hold:
\begin{equation}
\label{X vs Y for T=1}
\tag{$\ast$}
\| \partial_t^l \partial_x^\alpha u \|_{L^p(Q_{1}(x_0), \mu_{jp})} \, \lesssim \, (1+\sqrt{x_{0,n}})^{2j+1-|\alpha|} \, |Q_1(x_0)|^{\frac{1}{p}} \, \| f \|_{Y_p}
\end{equation}
and
\begin{equation}
\label{gradient u(1) vs Y}
\tag{$\ast\ast$}
|\nabla u(1,x_0)| \, \lesssim \,\| f \|_{Y_p} \, ,
\end{equation}
whenever $f$ is supported in $[0,1] \times \hs$. The general case follows by scaling under the coordinate transformation
\[
T_\lambda: [0,1] \times \hs \ni (t,x) \mapsto (\hat{t},\hat{x}) = (\lambda^2 t,\lambda x) \in [0,T] \times \hs
\]
for $0< \lambda \leq \sqrt{T}$. In this situation $u \circ T_\lambda$ is an energy solution on $[0,\lambda^{-2} T) \times \hs$ to $\lambda^2 (f\circ T_\lambda)$, see (\ref{linear scaling}). Then a calculation, using the geometric properties that have been discussed in section \ref{preliminaries}, implies
\[
\| \partial_{\hat{t}}^l \partial_{\hat{x}}^\alpha u \|_{L^p(Q_{\sqrt{\lambda}}(\hat{x}_0),\mu_{jp})} \, \lesssim \,
\lambda^{\frac{n+2}{p}+j-2l-|\alpha|} \, \| \partial_t^l \partial_x^\alpha (u \circ T_\lambda) \|_{L^p(Q_1(x_0), \, \mu_{jp})} \, .
\]
Applying (\ref{X vs Y for T=1}), combined with the fact that $|Q_1(x_0)| \sim \lambda^{n-2} \, |Q_{\!\sqrt{\lambda}}(\hat{x}_0)|$, the right hand side has the bound
\[
\lambda^{2l+\frac{|\alpha|}{2}-\frac{1}{2}} \, (\sqrt{\lambda} + \sqrt{\hat{x}_{0,n}})^{|\alpha|-2j-1} \, |Q_{\!\sqrt{\lambda}}(\hat{x}_0)|^{-\frac{1}{p}} \, \| \partial_{\hat{t}}^l \partial_{\hat{x}}^\alpha u \|_{L^p(Q_{\!\sqrt{\lambda}}(\hat{x}_0),\mu_{jp})}
\]
which is, due to (\ref{scaling of the Y-norm}), less than a constant times $ \| f \|_{Y_p}$. Now choosing $\lambda = r^2$ and taking the supremum over $r$ and $\hat{x}_0$ gives $\| u \|_{X_p^1} \lesssim \| f \|_{Y_p}$. \\
Likewise,
\[
|\nabla_{\!\hat{x}} u(\lambda^2,\lambda x_0)| \, = \, \lambda^{-1} \, |\nabla_{\!x} (u \circ T_\lambda)(1,x_0)| \, \overset{\text{(\ref{gradient u(1) vs Y})}}{\lesssim} \, \lambda \, \| f \circ T_\lambda \|_{Y_p} \, \lesssim \, \| f \|_{Y_p}
\]
for any $0<\lambda\leq\sqrt{T}$ and for almost every $x_0 \in \hs$. This proves the assertion for $g=0$. \\

Finally, if $u$ is a solution of $\partial_t u + Lu=f$ with $u(0)=g$, we write $u=u_1+u_2$, where $u_1$ solves the same equation with $u(0)=0$ and $u_2$ is a solution of the homogeneous initial value problem. Then 
\[
\| u \|_{X_p} \, \leq \, \| u_1 \|_{X_p} + \| u_2 \|_{X_p} \, \lesssim \, \| f \|_{Y_p} + \| g \|_{\dot{C}^{0,1}}
\]
which follows from the above estimate applied to $u_1$ and Corollary \ref{Lp estimates in the cylinder} applied to $u_2$. This yields the complete statement.
\end{proof}

It remains to provide the proof of the nonlinear estimate.

\begin{proof}[Proof of Lemma \ref{Y-norm vs X-norm}]
The proof requires a careful examination of the inhomogeneity coupled with the definition of the considered function spaces $Y_p$ and $X_p$. \\

First notice that
\begin{equation}
\label{Y-factor boundedness}
\tag{$\ast$}
R^3 \, (R+\sqrt{x_n})^{-1} \, \leq \, R^{\,4l+|\alpha|-1} \, (R+\sqrt{x_n})^{|\alpha|-2j-1}
\end{equation}
for any $(j, l, \alpha) \in \mathcal{CZ}$ such that the factor in the $Y_p$-norm can be bounded by each of the factors appearing in the $X_p^1$-norm. Now from Lemma \ref{nonlinearity} we know that the inhomogeneity can be written in the form
\[
f[u] \, = \, f_{0}[u] + x_n \, f_1[u] + x_n^{\,2} \, f_2[u]
\]
with
\[
\begin{split}
f_{0}[u] \, &= \, f_0^1(\nabla u) \star \nabla u \star D_{\!x}^2 u \, , \\
f_{1}[u] \, &= \, f_1^1(\nabla u) \star \nabla u \star D_{\!x}^3 u + f_1^2(\nabla u) \star P_2(D_{\!x}^2 u) \quad \text{and} \\
f_{2}[u] \, &= \, f_2^1(\nabla u) \star \nabla u \star D_{\!x}^4 u + f_2^2(\nabla u) \star D_{\!x}^2 u \star D_{\!x}^3 u + f_2^3(\nabla u) \star P_3(D_{\!x}^2 u) \, .
\end{split}
\]
By Remark \ref{decomposition of f_i^k} and the assumption $u \in \bar{B}_\eps^X$ with $0< \eps < \frac{1}{2}$, we have $|1+\partial_{x_n} u|^{-m} \leq (1- \eps)^{-m} < 2^m$ and hence
\begin{equation}
\label{f_i^k-boundedness}
\tag{$\ast\ast$}
\| f_i^k(\nabla u) \|_{L^{\infty}(I \times H)} \, < \, c(\eps)
\end{equation}
for every $i=0,1,2$ and $1\leq k\leq i+1$. \\

Moreover, since polynomial maps are analytic functions, we can expand each of the $f_i^k(\nabla u)$ into a power series that converges in $Y_p$ for every $u \in \bar{B}_\eps^X$. \\

We now consider each of the summands in the expression of $f[u]$ separately. For $f_i^1$ we observe that
\[
\| f_i^1(\nabla u) \nabla u D_{\!x}^{|\alpha|} u \|_{L^p(Q_{R}(x),\mu_{jp})} \, \leq \, \| f_i^k(\nabla u) \|_{L^{\infty}(I \times H)} \, \| \nabla u \|_{L^{\infty}(I \times H)} \, \| D_{\!x}^{|\alpha|} u \|_{L^p(Q_{R}(x),\mu_{jp})} \, ,
\]
where $j=0,1,2$ and $|\alpha|=j+2$. Using the inequalities obtained in (\ref{Y-factor boundedness}) and (\ref{f_i^k-boundedness}), we arrive at the estimate
\[
|Q_R(x)|^{-\frac{1}{p}} \, R^3 \, (R+\sqrt{x_n})^{-1} \, \| x_n^{\,j} \, f_i^1(\nabla u) \nabla u \, D_{\!x}^{|\alpha|} u \|_{L^p(Q_{R}(x))} \, \lesssim \, \| u \|_{X_p}^2 \, .
\]
The other parts of $f[u]$ can be estimated by means of H\"older's inequality and a localized version of the weighted Gagliardo-Nirenberg interpolation inequality: Let $j, k$ be any integers satisfying $0 \leq \theta = \frac{j}{k} < 1$, and $r,p,q, \sigma$ be any positive numbers with $\sigma > 0$ and $1\leq r, q, p \leq \infty$, respectively. Then
\[
\| D_{\!x}^j u \|_{L^r(\mu_\sigma)} \, \leq \, c(n,r) \, \| u \|_{L^q(\mu_\sigma)}^{1-\theta} \, \| D_{\!x}^{k} u \|_{L^p(\mu_\sigma)}^\theta \, ,
\]
where $\frac{1}{r} = \frac{1-\theta}{q} + \frac{\theta}{p}$. \\

Next we address the second part of the lemma, namely the one that includes the estimate for the difference $f[u_1]-f[u_2]$ for $u_1, u_2 \in \bar{B}_\eps^X$ with $0< \eps < \frac{1}{2}$. Expressed as a telescoping sum, this reads
\[
f_0[u_1] - f_0[u_2] \, = \,
\bigl( f_0^1(\nabla u_1) - f_0^1(\nabla u_2) \bigr) \star \nabla u_1 \star D_{\!x}^2 u_1 \, +
\] 
\[
+ \, f_0^1(\nabla u_2) \star \nabla(u_1 - u_2 ) \star D_{\!x}^2 u_1 + f_0^1(\nabla u_2) \star \nabla u_2 \star D_{\!x}^2(u_1 - u_2) \, ,
\]
and similarly we rewrite all the other terms except the last one. Here we have the following decomposition
\[
f_2^3(\nabla u_1) \star P_3(D_{\!x}^2 u_1) - f_{2}^3(\nabla u_2) \star P_3(D_{\!x}^2 u_2)
\]
\[ 
= \, \bigl( f_2^3(\nabla u_1) - f_2^3(\nabla u_2) \bigr) \star P_3(D_{\!x}^2 u_1) + f_2^3(\nabla u_2) \star \sum_{i=0}^2 P_{2-i}(D_{\!x}^2 u_1) \star D_{\!x}^2(u_1 - u_2) \star P_i(D_{\!x}^2 u_2) \, .
\]
But now we can apply the first part of the proof to each of these summands combined with the identity
\[
(1+\partial_{x_n} u_1)^{-m} - ( 1 + \partial_{x_n} u_2)^{-m}
\]
\[
= \, \frac{(\partial_{x_n} u_2 - \partial_{x_n} u_1)}{( 1 + \partial_{x_n} u_1)^m \, ( 1 + \partial_{x_n} u_2)^m} \; \sum_{i=1}^m \, \sum_{k=0}^{i-1} \binom{m}{i} \, (\partial_{x_n} u_1)^{i-1-k} \, (\partial_{x_n} u_2)^k
\]
to finish the proof of the lemma.
\end{proof}


\bigskip

\addsubsection*{Acknowledgments} The present work is based on the author's PhD thesis. He wishes to thank his doctoral advisor Herbert Koch for the introduction to the different theories that fit together so beautifully in this field, and for all the guidance and patience. This paper also benefited from helpful conversations with Clemens Kienzler and Felix Otto who provided valuable input to the writing. The author acknowledges support from the Bonn International Graduate School in Mathematics (BIGS), and the DFG through the Collaborative Research Center (SFB) 611 \emph{Singular Phenomena and Scaling in Mathematical Models}.

\end{document}